%% file: main_Arxiv2.tex
\documentclass[11pt]{amsart}
\usepackage{amsmath}
\usepackage{amsthm}
\usepackage{amsfonts}
\usepackage{amssymb}
\usepackage{epsfig}
\usepackage{color}
\usepackage{tikz}
\usepackage{enumitem} 
\usepackage{pgfplots}
\pgfplotsset{compat=newest}
\usetikzlibrary{plotmarks}
\usepackage{xcolor,colortbl}
\usepackage{multirow}
\usepackage{rotating}
\usepackage{booktabs}
\usepackage{amsmath}
\usepackage{appendix}
\usepackage{bm,bbm}
\usepackage[hidelinks]{hyperref}

\usepackage{mathtools}  
\newtheorem{theorem}{Theorem}

\newtheorem{corollary}[theorem]{Corollary}
\theoremstyle{definition}

\theoremstyle{lemma}
\newtheorem{lemma}[theorem]{Lemma}

\theoremstyle{remark}
\newtheorem{remark}[theorem]{Remark}

\numberwithin{theorem}{section}
\numberwithin{equation}{section}
\numberwithin{table}{section}
\numberwithin{figure}{section}
\input{def}
\allowdisplaybreaks
\begin{document}
\title[Time discretization schemes for hyperbolic PDEs on networks]{Time discretization schemes for hyperbolic systems on networks by $\eps$-expansion${}^{*}$}
\author[]{R.~Altmann$^\dagger$, C.~Zimmer$^{\ddagger}$}
\address{${}^{\dagger}$ Department of Mathematics, University of Augsburg, Universit\"atsstr.~14, 86159 Augsburg, Germany}
\address{${}^{\ddagger}$ Institute of Mathematics MA\,{}4-5, Technical University Berlin, Stra\ss e des 17.~Juni 136, 10623 Berlin, Germany}
\email{robert.altmann@math.uni-augsburg.de, zimmer@math.tu-berlin.de}
\thanks{${}^{*}$Research funded by the Deutsche Forschungsgemeinschaft (DFG, German Research Foundation) within the SFB 910, project number 163436311}
\date{\today}
\begin{abstract}
We consider partial differential equations on networks with a small parameter $\eps$, which are hyperbolic for $\eps>0$ and parabolic for $\eps=0$. With a combination of an $\eps$-expansion and Runge-Kutta schemes for constrained systems of parabolic type, we derive a new class of time discretization schemes for hyperbolic systems on networks, which are constrained due to interconnection conditions. 
For the analysis we consider the coupled system equations as partial differential-algebraic equations based on the variational formulation of the problem. We discuss well-posedness of the resulting systems and estimate the error caused by the $\eps$-expansion.
\end{abstract}
%
%
\maketitle
%
%
\section{Introduction}
The propagation of pressure waves in a network of pipes~\cite{Osi87, BroGH11} as well as the electro-magnetic-energy propagation along a network of transmission lines~\cite{MagWTA00, GoeHS16} can be modeled as a coupled system of hyperbolic partial differential equations (PDE). On each edge of the network (representing, e.g., a pipe or transmission line) we consider a one-dimensional linear wave system with damping of the form
\begin{align*}
\dot p^e(t,x) + a^e(x) p^e(t,x) + \partial_x m^e(t,x) &= \g^e(t,x), \\
\eps\, \dot m^e(t,x) + \partial_x p^e(t,x) + d^e(x) m^e(t,x) &= \f^e(t,x).
\end{align*}
Here, $p^e$ and $m^e$ model the potential and flow variables of the system on a single edge~$e$. The non-negative parameters $a^e$, $d^e$ include linear damping, e.g., due to friction, and the source terms $\g^e$, $\f^e$ may encode, e.g., the slope of a pipe. For the parameter $\eps$ we assume in this paper that $0 \le \eps \ll 1$. In a gas network this would equal the product of the adiabatic coefficient and the square of the Mach number. We emphasize that setting $\eps=0$ results in a parabolic PDE whereas the PDE is hyperbolic for $\eps>0$. 

Considering a network as illustrated in Figure~\ref{fig_network}, the stated conservation and balance laws require additional coupling conditions. 
These reflect important physical properties similarly to the circuit laws of Kirchhoff. In particular, we demand continuous potentials and that the sum of flows is balanced at each junction. Thus, the overall system combines hyperbolic PDEs with explicitly stated constraints, cf.~\cite{JanT14, EggKLMM18}. Mathematically, the resulting model equals a {\em partial differential-algebraic equation} (PDAE), see~\cite{LamMT13} for an introduction. \medskip
%
\begin{figure}
	\centering
	\begin{tikzpicture}[scale=0.90]
	\draw[thick, ->] (0, 0) -- (2.5, 0);
	\node at (1.5, 0.2) {$e_1$}; 
	\draw[thick, ->] (3, 0) -- (5.5, 0.9);
	\node at (4.4, 0.75) {$e_2$};
	\draw[thick, ->] (3, 0) -- (5.5, -0.9);
	\node at (4.4, -0.75) {$e_3$};
	\draw[thick, ->] (6, 1) -- (8.5, 1);
	\node at (7.4, 1.2) {$e_4$};
	\draw[thick, ->] (6, -1) -- (8.55, 0.75);
	\node at (7.4, -0.4) {$e_5$};
	\draw[thick, ->] (6, -1) -- (10.5, -1);
	\node at (8.6, -1.3) {$e_6$};
	\node at (0, 0) [circle, draw, fill=light_blue] {$v_1$}; 
	\node at (3, 0) [circle, draw, fill=light_blue] {$v_2$}; 
	\node at (6, 1) [circle, draw, fill=light_blue] {$v_3$}; 
	\node at (6, -1) [circle, draw, fill=light_blue] {$v_4$}; 
	\node at (9, 1) [circle, draw, fill=light_blue] {$v_5$};
	\node at (11, -1) [circle, draw, fill=light_blue] {$v_6$};
	\end{tikzpicture} 
	\caption{Example of a network of pipes, represeted by a graph with vertices $\V = \{v_1, v_2, v_3, v_4, v_5, v_6\}$ and edges $\E = \{ e_1 = (v_1,v_2), e_2 = (v_2,v_3), e_3 = (v_2,v_4), e_4 = (v_3,v_5), e_5 = (v_4,v_5), e_6 = (v_4,v_6)\}$ .}
	\label{fig_network}
\end{figure}
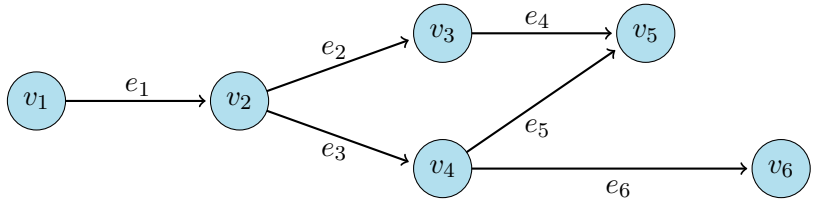
%

The aim of this paper is to derive time discretization schemes for the given class of PDAEs. Recently, time discretization schemes have been analyzed for the parabolic case ($\eps=0$) including Runge-Kutta methods~\cite{AltZ18} and discontinuous Galerkin methods~\cite{VouR18}. In general, one may say that the construction and analysis of numerical schemes need a combination of methods known from time-dependent PDEs, see, e.g.,~\cite{LubO95, EmmT10}, as well as strategies coming form the theory of differential-algebraic equations (DAE), cf.~\cite{HaiLR89, HaiW96, KunM06}.  

The foundation for the analysis of time discretization schemes are well-posedness results for the PDAE models. Hence, a major part of the paper discusses mild and classical solutions for different regularity assumptions on the right-hand side. For the particular case of gas networks with~$a^e=0$, well-posedness and exponential stability have been shown in~\cite{EggK18}. Therein, the connection to the parabolic system for $\eps=0$ is exploited. Similar techniques are applied in the present paper. \medskip

The paper is organized as follows. In Section~\ref{sect:prelim} we derive the model equations together with the constraints coming from coupling and boundary conditions. Further, we discuss two particular examples and needed function spaces. The variational formulation, which then leads to the considered PDAE system, is subject of Section~\ref{sect:formulation}. This includes uniqueness and existence results of mild and classical solutions. An important step towards the proposed time discretization scheme is then the consideration of the parabolic limit case, which we obtain for $\eps=0$. 

Section~\ref{sect:expansion} is devoted to the comparison of the two solutions coming from the hyperbolic and parabolic systems, respectively. For this, we consider the difference of the solutions for~$\eps>0$ and~$\eps=0$ and show that this is of order $\eps$ for appropriate initial values. In view of time discretization schemes of higher order, we also consider the second-order term in the $\eps$-expansion of the exact solution. We prove that this then yields a second-order approximation under certain additional conditions on the initial data. In Section~\ref{sect:index} we shortly comment on the differentiation index of the DAEs resulting from a spatial discretization of the coupled systems. Finally, we introduce in Section~\ref{sect:timeInt} time discretization schemes which result from a combination of the $\eps$-expansion of the solution and Runge-Kutta schemes. A numerical illustration of the results is given in Section~\ref{sect:numerics}.
%
%
\input{prelim} 				
\input{setting}  			
\input{expansion}
\input{semidiscrete} 		
\input{timeintegration}		
\input{numerics}			
%
%
\section{Conclusion}\label{sect_conclusion}
In this paper, we have derived time discretization methods for coupled systems of hyperbolic equations including a small parameter~$\eps>0$. An expansion of the solution in this parameter is analyzed and combined with Runge-Kutta methods applied to the limit equation for~$\eps=0$.  
The basis for the convergence proof is given by a number of existence results for the involved parabolic and hyperbolic PDAE models. These findings are in line with the results presented in~\cite{EggK18}, although they are based on a different weak formulation. Further, the present paper offers extensions in view of an additional damping term, the possibility of time-dependent right-hand sides, and higher order approximations in terms of~$\eps$. 
%
%
\section*{Acknowledgement}
Major parts of the paper were evolved at Mathematisches Forschungsinstitut Oberwolfach within a {\em Research in Pairs} stay in February 2018. We are grateful for the invitation and kind hospitality.
%
%
\newcommand{\etalchar}[1]{$^{#1}$}

\appendix
\newpage
\section{Collection of Equation Systems}
\noindent
The original hyperbolic system with coupling conditions \eqref{eqn_op_B_eps} has the form
\begin{alignat*}{5}
	\dot p &\ +\ &  a p &\ -\ &\calK^* m & + \Ce^*\lambda &\ =\ &\g -\Ci^*\r  &&\qquad \text{in }[H^1(\bE)]^*,\\
	\eps\, \dot m  &\ +\ &\calK p &\ +\ &d m &   &\ =\ & \f &&\qquad \text{in }[L^2(\E)]^*,\\
	& & \Ce p & & &  &\ =\ &  \h &&\qquad \text{in }\R^{\Nin}.
\end{alignat*}
Setting $\eps=0$, we obtain the constrained parabolic system~\eqref{eqn_op_B_noEps} for $(p_0, m_0, \lambda_0)$ with initial condition $p_0(0)=p(0)$, namely 
\begin{alignat*}{5}
	\dot p_0&\ +\ & a p_0&\ -\ &\calK^* m_0&\ +\ & \Ce^*\lambda_0 &= \g-\Ci^*\r  &&\qquad \text{in }[H^1(\bE)]^*,\\
	& & \calK p_0&\ +\ & d m_0& & &= \f  &&\qquad \text{in }[L^2(\E)]^*,\\
	& & \Ce p_0& & & &  &= \h &&\qquad \text{in } \R^{\Nin}.
\end{alignat*}
For a higher-order approximation we consider system~\eqref{eqn_op_B_approx_1} for $(p_1, m_1, \lambda_1)$ with $p_1(0)=0$,
\begin{alignat*}{5}
 \dot p_1 &\,+\, & a p_1 &\, - \, & \calK^* m_1 &\, +\, &\Ce^*\lambda_1 &= 0  &&\qquad \text{in }[H^1(\bE)]^*,\\
 & & \calK p_1 &\, +\, & d m_1 & & &= -\dot m_0  &&\qquad \text{in }[L^2(\E)]^*,\\
 & & \Ce p_1 &  & & & &= 0 &&\qquad \text{in } \R^{\Nin}.
\end{alignat*}
For the analysis we use the auxiliary problem~\eqref{eqn_op_B_station} with solution $(\pb,\mb,\lb)$, 
\begin{alignat*}{5}
	a \pb &\ -\ &\calK^* \mb & + \Ce^*\lb &\ =\ &-\Ci^*\r  &&\qquad \text{in }[H^1(\bE)]^*,\\
	\calK \pb &\ +\ &d \mb &   &\ =\ &  0 &&\qquad \text{in }[L^2(\E)]^*,\\
	\Ce \pb & & &  &\ =\ &  0 &&\qquad \text{in }\R^{\Nin}.
\end{alignat*}
The first-order differences $p-p_0$ and $m-m_0$ satisfy system~\eqref{eqn_inproof_pp0}, i.e.,  
\begin{alignat*}{5}
	\ddt (p-p_0)&\,+\,& a (p-p_0)&\,-\,& \calK^* (m-m_0) &\,=  0  &&\qquad \text{in } H^{-1}(\bE),\\
	&  & \calK (p-p_0)&\,+\,& d (m-m_0) &\,=   -\eps\, \dot m &&\qquad \text{in }[L^2(\E)]^*.
\end{alignat*}
The second-order differences $p-\hat p$ and $m-\hat m$ are solutions of the system
\begin{alignat*}{5}
 \ddt (p-\hat p) &\, + \,& a(p-\hat p) &\, -\, & \calK^* (m- \hat m) &= 0  &&\qquad \text{in }H^{-1}(\bE), \\
 & & \calK (p-\hat p)&\, +\, & d (m-\hat m) &= -\eps\, (\dot m - \dot m_0) &&\qquad \text{in }[L^2(\E)]^*
\end{alignat*}
In certain proofs we consider $\pt := p - \pb - \Ce^-\h$ and $\mt := m - \mb$, which satisfy system~\eqref{eqn_inproof_ph}, 
\begin{alignat*}{5}
\dot \pt &\ +\ & a\pt &\ -\ & \calK^* \mt &= \g-\dot \pb - a\Ce^- \h - \Ce^- \dot \h  && \qquad \text{in } H^{-1}(\bE),\\
\eps\, \dot \mt &\ +\ & \calK \pt &\ +\ & d \mt &= \f - \calK \Ce^- \h - \eps\, \dot \mb  && \qquad \text{in } [L^2(\E)]^*.
\end{alignat*}
%
%
\input{appendix_generalization}
%
%
\input{appendix_O-Bound}
%
\end{document}

%% file: prelim.tex
\section{Model Equations and Preliminaries}\label{sect:prelim}
In this section, we introduce the hyperbolic system equations as well as the coupling conditions, which are directly given by the underlying network structure. Furthermore, we introduce the needed function spaces and the constraint operators. These will be used in the subsequent section for the operator formulation of the system equations.  
%
%
\subsection{Network geometry}\label{sect:prelim:network}
We consider a directed graph $\G = (\V, \E)$, which encodes the geometry of the underlying network, cf.~Figure~\ref{fig_network} for an exemplary illustration. The set of vertices $\V$ includes the junctions of the network whereas the edges in $\E$ represent interconnections, e.g., pipes in gas networks or transmission lines in power networks. On each edge $e\in\E$ we consider a linear hyperbolic PDE in one space dimension of the form 
\begin{subequations}
\label{sys_equations}
\begin{align}
	\dot p^e + a^e p^e + \partial_x m^e &= \g^e, \\
	\eps \dot m^e + \partial_x p^e + d^e m^e &= \f^e
\end{align}
\end{subequations}
with appropriate boundary conditions and given initial values for $p^e(0)$ and $m^e(0)$. The variable $p^e$ models a {\em potential} whereas $m^e$ is a {\em flow variable}. The parameters $a^e$ and $d^e$ include damping to the system equations and are assumed to be constant in time. However, they may be space-dependent within its bounds 
\[
	0 \le a^e(x) \le \aM, \qquad
	0 < \dm \le d^e(x) \le \dM < \infty
\]
for almost every $x\in e$ and all~$e\in\E$. An essential assumption of this paper is that the parameter $\eps$ satisfies $\eps \ll 1$, which accounts for the different time scales of the two variables. We emphasize that setting $\eps=0$ results in a parabolic PDE whereas the PDE is of hyperbolic type for $\eps>0$. 

For the specification of boundary and coupling conditions we need to distinguish two kinds of vertices: The set of boundary vertices for which the potential variables are prescribed is denoted by~$\Vin$ and defines $\Nin:=|\Vin|\ge 1$. For the remaining nodes in $\Vout$ we either have coupling (at interior vertices) or boundary conditions for the flow variables. Further, we introduce $\Nout := |\Vout|$ such that $\Nin+\Nout$ equals the total number of vertices in $\G$. 
Given any vertex $v\in\V$, we define by $\E(v)$ the set of edges, which have $v$ as a vertex. Finally, the number $n^e(v) = \pm 1$ encodes the direction of the edge $e$. In particular, this means that the graph is directed from a vertex $v$ to  $v'$ if $n^e(v) = -n^e(v') = -1$. For a boundary node $v\in\V$ we assume that all adjacent edges are equally directed, i.e., the sign $n^e(v)$ is identical for all edges $e\in\E(v)$. Consequently, we either have only in- or only outflow at a boundary node. 

The system equations on the network $\G$ are now given edgewise by \eqref{sys_equations} in combination with certain boundary and coupling conditions. For this, we define the functions $p$ and $m$ elementwise by $p^e$ and $m^e$, respectively. Similarly, the damping parameters $a$ and $d$ are edgewise defined through $a^e$ and $d^e$, respectively. Note that this implies uniform bounds on $d$ in terms of $\dm$ and $\dM$. The coupling conditions are similar to the Kirchhoff's circuit laws, i.e., we assume that $p$ is continuous and that for every vertex $v \in \Vout$ we have 
\begin{align}
\label{constraint}
  \sum_{e\in \E(v)} n^e(v) m^e(v) = \r(v). 
\end{align}
The given function $\r\colon [0,T] \to \R^{\Nout}$ models the action of consumers, dissipation, or a lossless connection at interior junctions. Note that we use here an identification of $\R^{\Nout}$ and the nodes $\Vout$, which allows to evaluate $\r$ at a node $v\in \Vout$. 
The Dirichlet boundary conditions modeling the inflow are given by
\begin{align}
\label{boundary}
  p(v) = \h(v) 
\end{align} 
for every $v\in\Vin$ and prescribed $\h\colon [0,T] \to \R^\Nin$. As before, we use here an identification of $\R^\Nin$ and $\Vin$.
%
%
\subsection{Examples}
We discuss two particular applications, which lead to coupled systems of the given form with a small parameter $\eps$. 
\subsubsection*{Gas networks}
In the example of a gas network, each edge $e\in\E$ corresponds to a pipe in which we consider the propagation of pressure waves. Thus, the variables $p$ and $m$ stand for the pressure and the mass flux, respectively. Under certain simplifying assumptions this is governed by equations of the form~\eqref{sys_equations}, cf.~\cite{EggK17ppt}. In this case, the parameter~$\eps$ equals the product of the adiabatic coefficient and the square of the Mach number and is of order~$10^{-3}$, cf.~\cite{BroGH11}.
Further, we have $a^e=0$ and $d^e$ includes damping due to friction at the walls of the pipe. The right-hand side~$\f^e$ includes the slope of a pipe~$e$.

For the case that $\Vin$ includes all boundary nodes and $\r \equiv 0$ the existence of classical solutions has been discussed in~\cite{EggK17ppt}. The extension of this result to~$\r\neq 0$ is straightforward if~$\r$ is independent of time. In this paper, however, we allow more general boundary conditions and time-dependent right-hand sides, which includes the boundary data $\r$. 
\subsubsection*{Power networks} 
In a power network, each edge $e\in\E$ corresponds to one transmission line whereas a node models a customer, power supplier, or an interconnection. The corresponding variables then model the voltage ($p$) and the current ($m$). The hyperbolic equation, which is considered on a single edge is also called the {\em telegrapher's equation}, cf.~\cite{ MagWTA00, GoeHS16}. 
%
The power loss due to the resistance of the underlying material is modeled trough the parameter $d^e$. On the other hand, $a^e$ describes the capacitance of the transmission line and depends on its length. 
Also in this kind of models small values of~$\eps$ emerge. 
%
%
\subsection{Function spaces}
For a proper weak formulation of system~\eqref{sys_equations} we consider (piecewise) Sobolev spaces. 
First, we introduce the space $L^2(\E)$ consisting of all functions which are edgewise in $L^2(e)$. The corresponding inner product reads
\[
  (\cdot, \cdot)
  := (\cdot, \cdot)_{L^2(\E)}
  = \sum\nolimits_{e\in \E} (\cdot, \cdot)_{L^2(e)} 
\]
and defines the norm $\Vert \cdot\Vert := \Vert\cdot\Vert_{L^2(\E)} =  (\cdot, \cdot)^{1/2}$. The space of broken $H^1$-functions is denoted by $H^1(\E)$. This means that $v\in H^1(\E)$ if and only if $v|_e \in H^1(e)$ for all $e\in\E$. Due to well-known embedding theorems \cite[Ch.~21.3]{Zei90a}, such a function is continuous on all edges but may jump at junctions. 

For the subspace of globally continuous functions in $H^1(\E)$ we write $H^1(\bE)$. This notion is motivated by the fact that globally continuous functions, which are edgewise in $H^1$, are globally in $H^1$. We emphasize that the function space $H^1(\bE)$ is densely embedded in $L^2(\E)$. Further, $H_0^1(\bE)$ denotes the Sobolev space of functions in $H^1(\bE)$ with vanishing function values at the boundary nodes in $\Vin$. 
We emphasize, that this includes only the vertices for which we prescribe the potential $p$ in form of boundary conditions. Finally, we denote the dual space of $H_0^1(\bE)$ by $H^{-1}(\bE)$. 

As we consider time-dependent problems, appropriate solution spaces are given by Sobolev-Bochner spaces, cf.~\cite[Ch.~7]{Rou05}. 
Denoting the space of quadratic Bochner integrable functions with values taken in $V$ by $L^2(0,T;V)$, we use correspondingly the notion $H^m(0,T;V)$, $m\in\N$, for functions with higher regularity in time. Moreover, we define for two Sobolev spaces $V_1 \hookrightarrow V_2$ the space 
\[
  W(0,T;V_1,V_2) 
  := \big\{ v\in L^2(0,T;V_1)\ |\ \dot v \text{ exists in } L^2(0,T;V_2) \big\}. 
\]
%
%
%
\subsection{Constraint operators}\label{sect:prelim:constraints}
We introduce several constraint operators for the incorporation of the constraints on $m$, namely \eqref{constraint}, as well as the boundary conditions for $p$. 
%
First, we define the operator $\Di\colon H^1(\E)\to \R^{\Nout}$ by 
\[
  (\Di m)(v)
  := \sum\nolimits_{e\in \E(v)} n^e(v) m^e(v)  
\]
for a vertex $v\in\Vout$. Recall that we use a one-to-one correspondence of nodes in~$\Vout$ and the components of a vector in~$\R^{\Nout}$, as $\Nout$ equals the corresponding number of vertices. It is easy to see that the operator $\Di$ is surjective. Its dual operator $\Di^*\colon\R^{\Nout}\to [H^1(\E)]^*$ is defined through 
\[
  \Di^*\lambda 
 := \langle \Di^*\lambda, \cdot\,\rangle
  = (\Di\, \cdot\, , \lambda) 
  = (\Di\, \cdot\, )^T  \lambda
\]
for $\lambda \in \R^{\Nout}$. 

Second, we introduce $\De\colon H^1(\E)\to \R^{\Nin}$, which is defined similarly to $\Di$ but on the boundary nodes of the network with inflow. Thus, we set  
\[
  (\De p)(v)
  := \sum\nolimits_{e\in \E(v)} n^e(v) p^e(v) 
\]
for a boundary node $v\in\Vin$. Note that also $\De$ defines a surjective operator. The corresponding dual operator reads $\Ded\colon\R^{\Nin}\to [H^1(\E)]^*$ and is defined for $\lambda \in \R^\Nin$ through $\Ded\lambda = (\De\, \cdot\, , \lambda)$. 

In the following subsections we will also need function evaluations of $H^1(\bE)$-functions at the nodes of the network. We denote this operation by $\Ce\colon H^1(\bE)\to \R^{\Nin}$ for nodes with inflow and $\Ci\colon H^1(\bE)\to \R^{\Nout}$ for the remaining nodes. 
Obviously, also the operators $\Ci$ and $\Ce$ are surjective. The following lemma shows that all four operators are even inf-sup stable, since they map into a finite-dimensional space. 
\begin{lemma}
\label{lem_infsup}
Let $X$ be an arbitrary Banach space and the operator $\S \colon X \to \R^n$ linear, continuous, and surjective. Then $\S$ is inf-sup stable, meaning that there exists a positive constant $\beta>0$ with 
\[
  \adjustlimits\inf_{\mu\in \R^n}\sup_{v\in X} 
  \frac{(\S v, \mu)}{\Vert v\Vert_{X} \Vert \mu\Vert_{\R^n}} \ge \beta. 
\]
\end{lemma}
\begin{proof}
Let $\mu\in \R^{n} \setminus \{0\}$ be arbitrary and $e_i$ the $i$-th Cartesian unit vector in $\R^n$, $i=1,\ldots,n$. Since $\S$ is surjective, there exist linearly independent vectors $v_i\in X$ with $\S v_i = e_i$. For  $\tilde{v}:=\sum_{i=1}^n \mu_i v_i$ we then obtain the estimate 
\[ 
  \Vert \tilde{v} \Vert_{X} 
  \leq \sum_{i=1}^n |\mu_i| \Vert v_i \Vert_{X} 
  \leq \max_{i=1,\dots, n} \Vert v_i \Vert_{X} \sum_{i=j}^n |\mu_j|
  \leq \sqrt{n} \max_{i=1,\dots, n} \Vert v_i \Vert_{X}  \| \mu \|_{\R^n}.
\]
Therefore, it holds for $\S$ that
\[ 
  (\S \tilde{v}, \mu)
  = \sum_{i=1}^n  (\mu_i e_i) \cdot \mu 
  = \sum_{i=1}^n  \mu_i^2 
  = \|\mu\|^2_{\R^{n}} 
  \geq \frac{1}{\sqrt{n} \max_{i=1,\dots, n} \Vert v_i \Vert_{X}} \Vert \widetilde{v}\Vert_{X}\, \Vert \mu\Vert_{\R^n}. 
\]
Since $\mu$ was chosen arbitrarily, a lower bound for the inf-sup constant $\beta$ is given by $(\sqrt{n} \max_{i=1,\dots, n} \Vert v_i \Vert_{X})^{-1} >0$.  
\end{proof}
The shown inf-sub stability is a crucial property for the well-posedness of the PDAE in the weak form, as the incorporation of the constraints will lead to a saddle point structure of the system equations. 
Further, it implies the existence of right-inverses $\Ci^-\colon \R^{\Nout} \to H^1(\bE)$ and $\Ce^-\colon \R^{\Nin} \to H^1(\bE)$, which satisfy $\Ci\Ci^- = \id_{\R^\Nout}$ and $\Ce\Ce^- = \id_{\R^\Nin}$. 

In the following section we discuss the weak formulation of system~\eqref{sys_equations} including boundary and coupling conditions. Throughout this paper, we use for estimates the notion $a \lesssim b$ for the existence of a generic constant $c>0$ such that $a \le cb$.  

%% file: setting.tex
\section{Operator Formulation}\label{sect:formulation}
There are two possibilities for a weak formulation of the considered system \eqref{sys_equations}-\eqref{boundary}. First, one may consider the formulation, which is weak in $p$. This means that $p$ takes values in $L^2(\E)$ only and that the boundary conditions cannot be enforced pointwise. Second, we can assume that $p$ is continuous but the mass flux $m$ is in $L^2(\E)$. As a result, the coupling condition \eqref{constraint} cannot be formulated pointwise and needs to be incorporated in a weak sense. 

In this section, we show that the second approach is advantageous and pass over to the operator form of the system equations. Further, we discuss existence results of mild and classical solutions. We emphasize that, in this section, the property $\eps \ll 1$ is not of importance such that $\eps$ may be replaced by any other positive constant. 
In the parabolic limit case with $\eps=0$ we discuss the existence of weak solutions. 
%
%
\subsection{A first weak formulation}\label{sect:formulation:A}
In this subsection, we discuss the weak formulation with \eqref{constraint} enforced by means of an explicit constraint, cf.~\cite[App.~I]{EggKLMM18}. Assume for a moment that $p\in H^1(\bE)$ with $p(v) = \h(v)$ for $v\in\Vin$. Then, integration by parts yields for a test function $\u\in H^1(\E)$, 
\begin{align*}
  (\px p, \u)
  = \sum_{e\in\E} (\px p^e, \u^e)_e
  &= - \sum_{e\in\E} (p^e, \px\u^e)_e + \adjustlimits\sum_{v\in\V}\sum_{e\in\E(v)} p(v) \u^e(v) n^e(v) \\
  &= -(p, \px\u) + \sum_{v\in\Vout} (\Di\u)_v p(v) + ( \De\u, \h ).
\end{align*}
In this subsection, we only assume $p(t)\in L^2(\E)$ and introduce a Lagrange multiplier~$\kappa$ such that 
\[
  (\px p, \u)
  = -(p, \px\u) + \langle \Di^*\kappa, \u\rangle + \langle \Ded\h, \u \rangle.
\]  
With this, we obtain the following weak formulation: 
Given right-hand sides $\f\colon [0,T] \to [H^1(\E)]^*$, $\g\colon [0,T] \to L^2(\E)$, $\h\colon [0,T] \to \R^{\Nin}$, $\r\colon [0,T] \to \R^{\Nout}$ and initial data $p(0) \in L^2(\E)$ and $m(0) \in H^1(\E)$, we search for abstract functions  $p\colon[0,T]\to L^2(\E)$, $m\colon[0,T]\to H^1(\E)$, and $\kappa\colon[0,T]\to \R^{\Nout}$ such that 
\begin{subequations}
\label{eqn_weak_A_eps}
\begin{align}
  (\dot p, q) + (a p, q) + (\px m, q) &= (\g, q), \\
  \eps\, (\dot m, \u) - (p, \px \u) + (dm, \u) + \langle \Di^*\kappa, \u\rangle &= \langle \f, \u \rangle - \langle \Ded\h, \u\rangle, \\
  \langle \Di m, \mu\rangle &= ( \r, \mu )
\end{align}
\end{subequations}
for all test functions $q\in L^2(\E)$, $\u\in H^1(\E)$, and $\mu\in \R^{\Nout}$.
\begin{remark}
The partial derivative $\px$ applied to a function in $H^1(\E)$ denotes the edgewise derivative w.r.t.~the variable $x$.  
\end{remark}
\begin{remark}
We emphasize that the continuity of $p$, which is part of the classical formulation, is not part of the weak formulation \eqref{eqn_weak_A_eps}. Note, however, that already $p(t) \in H^1(\E)$ implies continuity at the interior vertices as well as the boundary conditions $p(t, v) = \h(t, v)$ for $v\in\Vin$ and all $t \in [0,T]$. 
\end{remark}
Within this paper, we will consider the alternative weak formulation, which is introduced in the following subsection. This is preferable, since it has a straightforward extension to multiple dimensions -- leading to a term $\ddiv (\nabla p)$ instead of $\nabla(\ddiv m)$ -- and is more robust in terms of $\eps$. 
Latter can be seen in the index analysis of the semi-discrete system, cf.~Section~\ref{sect:index}, and needed regularity assumptions in the limit case. Considering $\eps=0$ and a given initial value for $p_0$, we obtain the equation $d m_0(0) = \f(0) + \px p_0(0)$.  
This means that we need $p_0(0) \in H^1(\E)$ although we consider the weak formulation with $p$ taking values in $L^2(\E)$ only. 
%
%
\subsection{A second weak formulation}\label{sect:formulation:B}
The second possibility of a weak formulation allows the mass flux $m$ to take values in $L^2(\E)$. Since this disables us to explicitly enforce the constraint \eqref{constraint}, we include this only weakly. On the other hand, we model $p$ with values in $H^1(\bE)$ and thus, as a continuous function such that boundary conditions can be included explicitly in form of a constraint. 

To obtain a weak formulation, we again integrate by parts. More precisely, we derive for $m\in H^1(\E)$, which satisfies $\Di m  = \r$, and a test function $q\in H^1(\bE)$,   
\begin{align*}
  (\px m, q)
  = \sum_{e\in\E} (\px m^e, q^e)_e
  &= - \sum_{e\in\E} (m^e, \px q^e)_e + \adjustlimits\sum_{v\in\V}\sum_{e\in\E(v)} q(v) m^e(v) n^e(v) \\
  &= -(m, \px q) + \sum_{v\in\Vout} (\Di m)_v (\Ci q)_v + \sum_{v\in\Vin} (\De m)_v (\Ce q)_v \\
  &= -(m, \px q) + (\r, \Ci q) + (\De m, \Ce q).
\end{align*}
If we introduce a Lagrange multiplier in place of $\De m$, then we obtain the following weak formulation: Given right-hand sides $\f\colon [0,T]\to L^2(\E)$, $\g\colon [0,T]\to [H^1(\bE)]^*$, $\h\colon [0,T]\to \R^{\Nin}$, $\r\colon [0,T]\to \R^{\Nout}$ and initial data $p(0) \in H^1(\bE)$ and $m(0) \in L^2(\E)$, we search for $p\colon[0,T]\to H^1(\bE)$, $m\colon[0,T]\to L^2(\E)$, and $\lambda\colon[0,T]\to \R^{\Nin}$ such that 
\begin{subequations}
\label{eqn_weak_B_eps}
\begin{align}
	(\dot p, q) + (a p, q) - (m, \px q) + \langle \Ce^* \lambda, q\rangle  &= \langle \g, q\rangle - \langle \Ci^*\r, q\rangle, \\
	\eps\, (\dot m, \u) +  (\px p, \u) + (dm, \u) &= (\f, \u), \\
	\langle\Ce p, \mu\rangle &= (\h, \mu)
\end{align}
\end{subequations}
for all test functions $q\in H^1(\bE)$, $\u\in L^2(\E)$, and $\mu\in \R^{\Nin}$.
\begin{remark}
The derivation of the weak formulation~\eqref{eqn_weak_B_eps} shows that the Lagrange multiplier $\lambda$ has a physical interpretation, namely the weighted sum of the boundary values of $m$.	
\end{remark}
We obtain a more compact form of \eqref{eqn_weak_B_eps} if we consider corresponding operators in the dual spaces of $H^1(\bE)$, $L^2(\E)$, and $\R^{\Nin}$. For this, we introduce $\calK$ for the (edgewise) partial derivative. We understand this as an operator $\calK\colon H^1(\bE) \to [L^2(\E)]^*$ and it is defined for all $q\in H^1(\bE)$ through 
\[
  \langle \calK q, m\rangle 
  := (\px q, m). 
\]
This then leads to the operator formulation 
\begin{subequations}
\label{eqn_op_B_eps}
\begin{alignat}{5}
	\dot p &\ +\ &  a p &\ -\ &\calK^* m & + \Ce^*\lambda &\ =\ &\g -\Ci^*\r  &&\qquad \text{in }[H^1(\bE)]^*, \label{eqn_op_B_eps_a} \\
	\eps\, \dot m  &\ +\ &\calK p &\ +\ &d m &   &\ =\ & \f &&\qquad \text{in }[L^2(\E)]^*, \label{eqn_op_B_eps_b} \\
	& & \Ce p & & &  &\ =\ &  \h &&\qquad \text{in }\R^{\Nin}. \label{eqn_op_B_eps_c}
\end{alignat}
\end{subequations}
Before considering the question of the existence of solutions to~\eqref{eqn_op_B_eps}, we mention that the results of the following sections remain valid for certain generalizations, cf.~Appendix~\ref{app_generalization}.
This applies particularly to the damping terms and the constraint equation~\eqref{eqn_op_B_eps_c}.  
%
%
\subsection{Existence of solutions}
It turns out that the following auxiliary problem is helpful for the upcoming analysis, 
\begin{subequations}
\label{eqn_op_B_station}
	\begin{alignat}{5}
		a \pb &\ -\ &\calK^* \mb & + \Ce^*\lb &\ =\ &-\Ci^*\r  &&\qquad \text{in }[H^1(\bE)]^*, \label{eqn_op_B_station_a}\\
		\calK \pb &\ +\ &d \mb &   &\ =\ &  0 &&\qquad \text{in }[L^2(\E)]^*, \label{eqn_op_B_station_b}\\
		\Ce \pb & & &  &\ =\ &  0 &&\qquad \text{in }\R^{\Nin}. \label{eqn_op_B_station_c}
	\end{alignat}
\end{subequations}
Note that the system does not include derivatives of the variables but that the right-hand side may still be time-dependent. To show the existence of a solution $(\pb,\mb,\lb)$ we first consider the time-independent case. For this, we define $\wlap$ as the weak Laplacian. 
\begin{lemma}
\label{lem_wlap}
The operator $\wlap:= \calK^\ast d^{-1} \calK \colon H^1(\bE) \to [H^1(\bE)]^\ast$ is linear, continuous, and non-negative. Furthermore, its restriction to $H^1_0(\bE)$ is elliptic, i.e., there exists a constant $\cLap>0$ such that for all $q \in H^1_0(\bE)$ it holds that
\[ 
  \langle \wlap q , q \rangle \ge \cLap\, \| q\|^2_{H^1(\bE)}. 
\]
\end{lemma}
\begin{proof}
As $\wlap$ is the weak Laplacian with an additional non-negative coefficient $d$, the properties are easy to show. The ellipticity follows by the Poincar{\'e}-Friedrichs inequality. 
\end{proof}
\begin{lemma}[Constant right-hand sides]
\label{lem_op_B_station}
For given $\r \in \R^{\Nout}$, system~\eqref{eqn_op_B_station} has a unique solution
$$ \big(\pb,\mb,\lb \big) \in H^1_0(\bE) \times L^2(\E) \times \R^{\Nin},$$
which depends linearly and continuously on the data, i.e., $ \|\pb\|_{H^1(\bE)} + \|\mb\| + | \lb | \lesssim  |\r |$.
\end{lemma}
\begin{proof}
Since equation~\eqref{eqn_op_B_station_b} is stated in $L^2(\E)$, we can insert this equation into~\eqref{eqn_op_B_station_a}, which results in 
\begin{alignat*}{5}
	(\wlap + a)\pb&\ + \Ce^*\lb &\ =\ &-\Ci^*\r  &&\qquad \text{in } [H^1(\bE)]^*,\\
	\Ce \pb&  &\ =\ & 0 &&\qquad \text{in } \R^{\Nin}.
\end{alignat*}
%
By standard arguments~\cite[Ch.~II.1.1]{BreF91} this system has a unique solution $\pb \in H^1_0(\bE)$, $\lambda \in \R^{\Nin}$, which is bounded in terms of $r$. The existence of $\mb$ and the stability bound then follow by $\mb = - d^{-1} \calK \pb$.
\end{proof}
As an immediate consequence we get the following existence result for system~\eqref{eqn_op_B_station} with a time-dependent right-hand side.
\begin{corollary}
\label{cor_op_B_station}
Consider $\r \in H^m(0,T;\R^{\Nout})$ for some $m\in\N$. Then, system~\eqref{eqn_op_B_station} has a unique solution 
\[ 
	\big(\pb,\mb,\lb \big) 
	\, \in\,  H^m(0,T; H^1_0(\bE)) \times H^m(0,T; L^2(\E)) \times H^m(0,T;\R^{\Nin}),
\]
which is bounded in terms of~$\r$.
\end{corollary} 
\begin{proof}
The existence follows directly by Lemma~\ref{lem_op_B_station}, if we consider system~\eqref{eqn_op_B_station} pointwise in time. The resulting solution is $H^m$-regular in time, since the operators are time-independent.
\end{proof}
We return to system~\eqref{eqn_op_B_eps}. To show the existence of mild and classical solutions we need the following lemma. 
\begin{lemma}\label{lem_unbounded_A}
Consider the (unbounded) operator 
\begin{equation}
\label{eqn_unbounded_A}
	\A := \begin{bmatrix}
	-\beta a & \hphantom{-}\calK^\ast \\	
	-\calK & -d/\beta
	\end{bmatrix}\colon 
	D(\A) \subset \big(L^2(\E) \times L^2(\E)\big) \to L^2(\E) \times L^2(\E)
\end{equation} 
for arbitrary positive $\beta>0$ and the domain
\[
  D(\A)
  = H^1_0(\bE) \times \big\{ m \in L^2(\E) \,|\ \exists\, \mstar \in L^2(\E)\colon (\mstar, q) = \langle \calK^* m, q\rangle\ \text{for all } q\in H^1_0(\bE) \big\}. 
\]  
Then, $\A$ generates a $C_0$-semigroup.
\end{lemma}
\begin{proof}
We show that $\A$ is a densely defined, closed, and dissipative operator with a dissipative adjoint~$\A^\ast$. The statement then follows by~\cite[Ch.~1.4, Cor.~4.4]{Paz83}. The operator is densely defined, since $D(\A)$ contains all functions $(p,m)$ with $p|_e, m|_e \in C^\infty_c(e)$ for all $e\in \E$. The closeness follows by the fact, that the operator $\A$ defines an invertible operator
\[
	\A\colon 
	H^1_0(\bE) \times L^2(\E) 
	\to H^{-1}(\bE) \times [L^2(\E)]^\ast
\]
with bounded inverse, cf. the proof of Lemma~\ref{lem_op_B_station}. For the dissipativity we note that 
\[
	\A^\ast 
	= \begin{bmatrix}
	-\beta a & -\calK^\ast \\
	\calK &  -d/\beta
	\end{bmatrix} 
\] 
has the same domain as $\A$. Let  $(p,m)\in D(\A)$ be arbitrary and $(p^\ast,m^\ast) \in [L^2(\E)]^\ast \times [L^2(\E)]^\ast$ the image of $(p,m)$ under the isometric Riesz isomorphism~\cite[Ch.~18.11]{Zei90a}. Then, it holds that
\begin{align*}
	\big\langle (p^\ast,m^\ast), \A(p,m) \big\rangle 
	 &= - \beta (p,a p) + (p,\calK^\ast m ) - (m,\calK p) - \tfrac{1}{\beta} (m,d m)\\
	 &= - \beta (p,a p) - \tfrac{1}{\beta} (m,d m) \leq 0.
\end{align*}
Hence, the operator $\A$ is dissipative. The dissipativity of $\A^\ast$ follows similarly.  
\end{proof}
\begin{lemma}[Existence of a mild solution]
\label{lem_mild_sol_B}
Consider right-hand sides $\f, \g\in L^2(0,T;L^2(\E))$, $\h \in H^1(0,T; \R^{\Nin})$, and $\r\in H^1(0,T; \R^{\Nout})$. Further assume initial data $p(0) \in L^2(\E)$ and $m(0)\in L^2(\E)$. Then, there exists a unique mild solution $(p,m,\lambda)$ of~\eqref{eqn_op_B_eps} with
$$
	p\in C([0,T],L^2(\E)) \cap H^1(0,T;H^{-1}(\bE)) \quad \text{ and } \quad m\in C([0,T],L^2(\E)).
$$
The Lagrange multiplier~$\lambda$ exists in a distributional sense with a regular primitive in the space~$C([0,T],\R^{\Nin})$ and it holds that 
\[ 
	\dot{p} + \Ce^\ast \lambda \in L^2(0,T;[H^1(\bE)]^\ast).
\]
\end{lemma}
\begin{proof}
Let $\pb \in H^1(0,T;H^1(\bE))$, $\mb \in H^1(0,T;L^2(\E))$, and $\lb \in H^1(0,T;\R^{\Nin})$ be the unique solution of system~\eqref{eqn_op_B_station}, cf.~Corollary~\ref{cor_op_B_station}. The introduction of $\pt := p - \pb - \Ce^- h$, $\mt := m - \mb$, and $\lt := \lambda - \lb$ leads together with~\eqref{eqn_op_B_eps} to the system 
\begin{alignat*}{5}
	\dot \pt &\ +\ & a\pt &\ -\ & \calK^* \mt &\ +\ & \Ce^*\lt &= \g-\dot \pb - a\Ce^- \h - \Ce^- \dot \h  && \qquad \text{in }[H^1(\bE)]^*, \\
	\eps\, \dot \mt &\ +\ & \calK \pt &\ +\ & d \mt &  &  &= \f - \calK \Ce^- \h - \eps\, \dot \mb  && \qquad \text{in }[L^2(\E)]^*, \\
	& & \Ce \pt & & & & &= 0 && \qquad \text{in } \R^{\Nin} 
\end{alignat*}
with initial values $\pt(0) = p(0)-\pb(0) - \Ce^- h(0) \in L^2(\E)$ and $\mt(0) = m(0) - \mb(0) \in L^2(\E)$. Since $\Ce \pt = 0$, we can reduce the evolution equation to
\begin{subequations}
\label{eqn_inproof_ph}
\begin{alignat}{5}
\dot \pt &\ +\ & a\pt &\ -\ & \calK^* \mt &= \g-\dot \pb - a\Ce^- \h - \Ce^- \dot \h  && \qquad \text{in } H^{-1}(\bE), \label{eqn_inproof_ph_a}\\
\eps\, \dot \mt &\ +\ & \calK \pt &\ +\ & d \mt &= \f - \calK \Ce^- \h - \eps\, \dot \mb  && \qquad \text{in } [L^2(\E)]^*. \label{eqn_inproof_ph_b}
\end{alignat}
\end{subequations}
Now consider the state $x := [\tfrac{1}{\sqrt{\eps}} \pt, \mt]^T$. Then, equation~\eqref{eqn_inproof_ph} becomes the abstract Cauchy problem 
\begin{subequations}
\label{eqn_cauchy_on_ker}
\begin{align}
	\dot{x}&= \tfrac{1}{\sqrt{\eps}} \A x + F = \tfrac{1}{\sqrt{\eps}} \A x + \begin{bmatrix}
	\tfrac{1}{\sqrt{\eps}} (\g-\dot \pb - a\Ce^-\h - \Ce^- \dot \h)\\
	\tfrac{1}{\eps} (\f - \calK \Ce^- h) - \dot \mb
	\end{bmatrix}, \label{eqn_cauchy_on_ker_a}\\
	x(0) &= \big[\tfrac{1}{\sqrt{\eps}} \pt(0),\, \mt(0) \big]^T \label{eqn_cauchy_on_ker_b}
\end{align}
\end{subequations}
with the operator $\A$ from Lemma~\ref{lem_unbounded_A} and $\beta = \sqrt{\eps}$. Note that we used the density of $H^1_0(\bE)$ in $L^2(\bE)$. 
Since the right-hand side satisfies $F\in L^2(0,T;L^2(\E)\times L^2(\E))$ and $x(0) \in L^2(\E)\times L^2(\E)$, the Cauchy problem has a unique mild solution $x \in C([0,T],L^2(\E)\times L^2(\E))$. Thus, $p = \pt + \pb + \Ce^-h \in C([0,T], L^2(\E))$ has a derivative in $L^2(0,T;H^{-1}(\E))$ by~\eqref{eqn_inproof_ph_a} and $m=\mt + \mb$ is an element of $C([0,T], L^2(\E))$. Finally, $\lambda$ can be constructed as in the proof of~\cite[Th~3.3]{EmmM13}.
\end{proof}
Considering more regularity for the given data, we now show the existence of a classical solution. For this, we again analyze the corresponding Cauchy problem. 
\begin{lemma}[Existence of a classical solution]
\label{lem_clas_sol_B}
Let the right-hand sides satisfy $\f, \g\in H^1(0,T;L^2(\E))$, $\h \in H^2(0,T; \R^{\Nin})$, and $\r\in H^2(0,T; \R^{\Nout})$. Further assume consistent initial data $p(0) \in H^1(\bE)$ and $m(0)\in L^2(\E)$, i.e., $\Ce p(0)=h(0)$, and the existence of a function $\mstar \in L^2(\E)$ with 
\begin{align}
\label{def_mstar}
  (\mstar, q) 
  = \langle \calK^* m(0),\, q\rangle - \langle \Ci^\ast r(0),\, q\rangle
\end{align}
for all $q\in H^1_0(\bE)$. Then, there exists a unique classical solution $(p,m,\lambda)$ of~\eqref{eqn_op_B_eps} with
\begin{align*}
	p\in C([0,T],H^1(\bE)) \cap C^1([0,T],L^2(\E)), &&m\in C^1([0,T],L^2(\E)), &&\lambda \in C([0,T], \R^{\Nin}).
\end{align*}
\end{lemma}
\begin{proof}
Following the proof of Lemma~\ref{lem_mild_sol_B} under the given assumptions, we notice that the right-hand side of the Cauchy problem~\eqref{eqn_cauchy_on_ker} is an element of~$H^1(0,T;L^2(\E)\times L^2(\E))$. Further we have $x(0) \in D(\A)$, since $\pt(0) \in H^1(\bE)$ and $\mt(0)$ satisfies, due to \eqref{eqn_op_B_station_a}, 
$$
	\calK^\ast \mt(0)
	= \calK^\ast m(0) - \calK^\ast \mb(0)
	= \calK^\ast m(0) - \Ci^\ast r(0) 
	= \mstar
	\quad\text{in } H^{-1}(\bE).
$$
The claimed solution spaces of~$p$ and~$m$ follow by~\cite[Ch.~4, Cor.~2.10]{Paz83}. Since the operator $\Ce$ satisfies an inf-sup condition by Lemma~\ref{lem_infsup}, there exists a unique (and continuous) multiplier~$\lambda$ which satisfies~\eqref{eqn_op_B_eps_a}, cf.~\cite[Lem.~III.4.2]{Bra07}. 
\end{proof}
\begin{remark}
A sufficient condition for the existence of~$\mstar \in L^2(\E)$ as stated in Lemma~\ref{lem_clas_sol_B} is that the initial value~$m(0) \in H^1(\E)$ satisfies~$\Di m(0) = \r(0)$. In this case, the function~$\mstar$ equals the piece-wise partial derivative of $m(0)$.
\end{remark}
The proofs of Section~\ref{sect:expansion} use certain estimates for the mild solution introduced in Lemma~\ref{lem_mild_sol_B}. For this, we first consider the classical solution from Lemma~\ref{lem_clas_sol_B} and $(\pt, \mt)$ as defined in~\eqref{eqn_inproof_ph}. Using $\pt$ and $\mt$ as test functions in~\eqref{eqn_inproof_ph} and integrating over the time interval $[0,t]$, we obtain 
%
%
\begin{align}
\Vert \pt(t)\Vert^2& + \eps\, \Vert \mt(t) \Vert^2 + \dm \int_0^t\Vert \mt (s)\Vert^2 \ds \notag\\
\le\ &  \Vert \pt(0) \Vert^2 + \eps\, \Vert \mt(0)\Vert^2 + \int_0^t \Vert (\g - \dot{\pb} - a\Ce^- \h - \Ce^-\dot\h)(s)\Vert^2 \label{eqn_estimate_clas} \\ 
&\hspace{5.2cm}+ \Vert (\f - \calK \Ce^- h)(s) \Vert^2 + \frac{\eps^2}{\dm} \Vert \dot{\mb}(s) \Vert^2  \ds \notag \\
\lesssim\ & \Vert \pt(0) \Vert^2 + \eps\, \Vert \mt(0)\Vert^2 + \int_0^t \Vert \f(s) \Vert^2 + \Vert \g(s) \Vert^2 +  |\h(s)|^2 + |\dot \h(s)|^2 + |\dot \r(s)|^2\ds.
\notag
\end{align}
Note that we have used Corollary~\ref{cor_op_B_station} in the second step. By the continuity of the semigroup generated by $\A$ and the density of $D(\A)$ in $L^2(\E) \times L^2(\E)$ as well as of $H^{\ell+1}(0,T;\X)$ in $H^{\ell}(0,T;\X)$ for $\X\in \{ L^2(\bE), \R^\Nin, \R^\Nout\}$ and $\ell\ge 0$, estimate~\eqref{eqn_estimate_clas} is also satisfied for the mild solution $(\pt,\mt)$ of~\eqref{eqn_inproof_ph} under the conditions of Lemma~\ref{lem_mild_sol_B}. 
%
%
\subsection{Parabolic limit case}
For the numerical simulation of the hyperbolic network equations, we will exploit the parabolic structure of the limit equation for $\eps=0$. The corresponding solution is denoted by $(p_0, m_0, \lambda_0)$ and solves the system 
\begin{subequations}
\label{eqn_op_B_noEps}
\begin{alignat}{5}
	\dot p_0&\ +\ & a p_0&\ -\ &\calK^* m_0&\ +\ & \Ce^*\lambda_0 &= \g-\Ci^*\r  &&\qquad \text{in }[H^1(\bE)]^*, \label{eqn_op_B_noEps_a}\\
	& & \calK p_0&\ +\ & d m_0& & &= \f  &&\qquad \text{in }[L^2(\E)]^*, \label{eqn_op_B_noEps_b}\\
	& & \Ce p_0& & & &  &= \h &&\qquad \text{in } \R^{\Nin}. \label{eqn_op_B_noEps_c}
\end{alignat}
\end{subequations}
The initial condition is given by $p_0(0) = p(0)$. Again, we need to discuss the existence of solutions. 
\begin{lemma}[Existence of a weak solution $(p_0,m_0)$]
\label{lem_p0m0}
Consider system~\eqref{eqn_op_B_noEps} with right-hand sides $\f \in L^2(0,T;L^2(\E))$, $\g\in L^2(0,T;[H^1(\bE)]^\ast)$, $\h \in H^1(0,T;\R^{\Nin})$,  and $\r \in L^2(0,T;\R^{\Nout})$. Assume further that $p_0(0)\in L^2(\E)$. Then, there exists a unique weak solution with
\[
	p_0 \in L^2(0,T;H^1(\bE)) \cap C([0,T], L^2(\E)), \qquad 
	m_0 \in L^2(0,T;L^2(\E)),
\]
satisfying the initial condition in $L^2(\E)$. The Lagrange multiplier~$\lambda_0$ exists in a distributional sense with 
\[
	\dot{p}_0 + \Ce^*\lambda_0 \in L^2(0,T;[H^1(\bE)]^\ast).
\]
\end{lemma}
\begin{proof}
Since the algebraic equation~\eqref{eqn_op_B_noEps_b} is stated in $L^2(\E)$, we can insert this equation into~\eqref{eqn_op_B_noEps_a}, which results in 
\begin{subequations}
\label{eqn_op_B_noEps_red}
\begin{alignat}{5}
	\dot p_0 + (\wlap + a) p_0& + \Ce^*\lambda_0 &\ =\ & \g-\Ci^*\r +  \calK^\ast d^{-1} \f &&\qquad \text{in }[H^1(\bE)]^*,\label{eqn_op_B_noEps_red_1}\\
	\Ce p_0&  &\ =\ & \h &&\qquad \text{in } \R^{\Nin} \label{eqn_op_B_noEps_red_3}
\end{alignat}
\end{subequations}
with the operator $\wlap$ introduced in Lemma~\ref{lem_wlap}. The existence of a unique partial solution $(p_0,\lambda_0)$ follows then by~\cite[Th.~3.3]{EmmM13}, using that $H^1_0(\bE)$ is dense in $L^2(\E)$. Finally, with equation~\eqref{eqn_op_B_noEps_b} the mass flow is given by
\begin{equation}\label{eqn_mass_flow_noEps}
	m_0 = d^{-1}(\f - \calK p_0)
\end{equation}
and therefore unique and an element of $L^2(0,T;L^2(\E)).$
\end{proof}
For the analysis in Section~\ref{sect:expansion} we also need solutions with higher regularity including an appropriate multiplier~$\lambda_0$. For this, we consider more regular right-hand sides and initial values. 
\begin{lemma}[Weak solution with higher regularity]
\label{lem_weak_op_B_noeps_reg} 
Consider right-hand sides $\f \in H^1(0,T;L^2(\E))$, $\g \in W(0,T;[H^1(\bE)]^\ast,H^{-1}(\bE))$, $\h \in H^2(0,T;\R^{\Nin})$, and~$\r \in H^1(0,T;\R^{\Nout})$. Further assume consistent initial data $p_0(0) \in H^1(\bE)$ with $\Ce p_0(0) = \h(0)$ and the existence of a function $\pstar\in L^2(\E)$ such that 
\begin{align}\label{def_pstar}
  (\pstar, q) 
  = \big\langle \g(0) -\Ci^\ast \r(0) - (\wlap +a) p_0(0) - \Ce^- \dot h (0) + \calK^\ast d^{-1}\f(0), q \big\rangle
\end{align}
for all $q\in H^1_0(\bE)$.
Then the solution of system~\eqref{eqn_op_B_noEps} satisfies 
\begin{gather*}
	p_0 \in  H^1(0,T; H^1(\bE)) \cap C^1([0,T],L^2(\E)), \quad
	m_0 \in  H^1(0,T;L^2(\E)), \quad
	\lambda_0 \in C([0,T],\R^{\Nin}).
\end{gather*}
\end{lemma}
\begin{proof}
We introduce $\pt_0 := p_0 -\Ce^-\h$ with $\Ce^-\h \in H^2(0,T;H^1(\bE))$ by assumption. Due to~\eqref{eqn_op_B_noEps_red}, this functions satisfies $\Ce\pt_0 = 0$ and  
\begin{equation}
\label{eqn_op_B_noEps_pot}
  \dot \pt_0 + \wlap \pt_0 
  = \g-\Ci^*\r - (\wlap + a) \Ce^-\h - \Ce^-\dot \h +  \calK^\ast d^{-1} \f \qquad \text{in }H^{-1}(\bE)
\end{equation}
with initial value $\pt_0(0)=p_0(0)-\Ce^- \h(0)$. The given assumptions imply that the right-hand side is an element of $H^1(0,T;H^{-1}(\bE))$ and 
$
  \dot \pt_0(0) 
  = \pstar 
  \in L^2(\E)
$. 
By \cite[Th.~IV.27.2]{Wlo92} we conclude that $\pot \in H^1(0,T;H^1_0(\bE))\cap C^1([0,T],L^2(\E))$. Finally, the stated regularity of $m_0$ follows by equation~\eqref{eqn_mass_flow_noEps} and, since there is a bounded left-inverse of~$\Ce^\ast$, the existence of~$\lambda_0$ by~\cite[Lem.~III.4.2]{Bra07}.
\end{proof}
After we have discussed the existence of solutions to systems~\eqref{eqn_op_B_eps} and~\eqref{eqn_op_B_noEps}, we compare these two solutions in the following section. 

%% file: expansion.tex
\section{$\eps$-Expansion of the Solution}\label{sect:expansion}
Assuming $\eps \ll 1$, we consider an expansion of the variables $p$ and $m$ in $\eps$, i.e., 
\begin{align}
\label{eqn_epsExp}
  p(t,x) = p_0(t,x) + \eps p_1(t,x) + \dots, \qquad
  m(t,x) = m_0(t,x) + \eps m_1(t,x) + \dots.
\end{align}
In the same manner, the Lagrange multiplier $\lambda$ is decomposed into $\lambda = \lambda_0 + \eps\lambda_1 + \dots$. As in the previous sections, the variables $p_j$, $m_j$, and $\lambda_j$ are defined edgewise by $p^e_j$, $m^e_j$, and $\lambda^e_j$ for all $e\in\E$. We are now interested in the approximation properties of $p_0$ and $m_0$ as well as of 
\[
  \hat p(t,x) := p_0(t,x) + \eps p_1(t,x), \qquad
  \hat m(t,x) := m_0(t,x) + \eps m_1(t,x).
\]
%
%
\subsection{First-order approximation}\label{sect:expansion:first}
We first discuss the approximation property of the pair $(p_0,m_0)$, which solves together with the multiplier $\lambda_0$ system~\eqref{eqn_op_B_noEps}. It was already shown in~\cite[Th.~1]{EggK17ppt} that this approximation is of order $\sqrt\eps$ and -- under certain assumptions on the initial data -- of order $\eps$. The proof, however, is based on the stationary solution of the system, which is not applicable in the present case of time-dependent right-hand sides and non-trivial boundary conditions. We present here an alternative proof, which also applies to the weak solution. 

We consider the difference of systems~\eqref{eqn_op_B_eps} and~\eqref{eqn_op_B_noEps}. Since $p$ and $p_0$ satisfy the same boundary conditions, we can neglect the constraint and restrict the test functions in the first equation to $H^1_0(\bE)$. This then leads to 
\begin{subequations}
\label{eqn_inproof_pp0}
	\begin{alignat}{5}
	 	\ddt (p-p_0)&\,+\,& a (p-p_0)&\,-\,& \calK^* (m-m_0) &\,=  0  &&\qquad \text{in } H^{-1}(\bE), \label{eqn_inproof_pp0_a}\\
	 	&  & \calK (p-p_0)&\,+\,& d (m-m_0) &\,=   -\eps\, \dot m &&\qquad \text{in }[L^2(\E)]^*. \label{eqn_inproof_pp0_b}
	\end{alignat}
\end{subequations}
The initial condition satisfies $(p-p_0)(0) = 0$. To derive estimates, we follow two approaches. First, we consider~\eqref{eqn_inproof_pp0} as a parabolic system with a right-hand side $\dot m$, which leads to Theorem~\ref{thm_p0m0}. Second, we will subtract $\eps \dot m_0$ from the second equation and consider the result as a hyperbolic system. This will be subject of Theorem~\ref{thm_p0m0_2} below. 
\begin{theorem}[First-order approximation I]
\label{thm_p0m0}
Assume right-hand sides $\f, \g \in H^1(0,T;L^2(\E))$, $\h \in H^2(0,T; \R^\Nin)$, and $\r \in H^2(0,T; \R^\Nout)$. Further, assume consistent initial data $p(0)\in H^1(\bE)$ and $m(0) \in L^2(\E)$ such that there is a function $\mstar\in L^2(\E)$ satisfying~\eqref{def_mstar}. Then, the difference of $(p,m)$ and $(p_0, m_0)$ is bounded for $t\le T$ by
\begin{align*}
  \Vert (p-p_0)(t) \Vert^2 &+ \cLap \int_0^t \Vert (p-p_0)(s) \Vert^2_{H^1(\bE)} \ds  + \dm \int_0^t \Vert (m-m_0)(s) \Vert^2 \ds \\
  &\quad\qquad \le \eps\, \frac{1}{\dm^2}\Big(1+\frac{1}{\cLap \dm}\Big)\, \Vert \calK p(0) - f(0) + d m(0) \Vert^2 + \eps^2\,  C_\text{data}
\end{align*}
with a constant 
\[
  C_\text{data} 
  = C_\text{data} \big( \Vert \mstar \Vert^2_{L^2(\E)}, \Vert \f\Vert^2_{H^1(0,T;L^2(\E))}, \Vert \g\Vert^2_{H^1(0,T;L^2(\E))}, \Vert \h\Vert^2_{H^2(0,T;\R^\Nin)}, \Vert \r\Vert^2_{H^2(0,T;\R^\Nout)} \big).
\]
\end{theorem}
\begin{proof}
We consider $p-p_0$ as test function in \eqref{eqn_inproof_pp0_a} and $m-m_0$ as test function \eqref{eqn_inproof_pp0_b}. Adding and integrating the resulting equations, we obtain by the non non-negativity of~$a$ and Young's inequality 
\begin{align}
\label{eqn_inproof_p0m0}
  \Vert (p-p_0)(t) \Vert^2 + 2 \dm \int_0^t \Vert (m-m_0)(s) \Vert^2 \ds 
  &\le \frac{\eps^2}{2 \dm} \int_0^t \Vert \dot m(s)\Vert^2 \ds.
\end{align}
On the other hand, using $p-p_0$ and $\frac{1}{d} \calK(p-p_0)$ as test functions in~\eqref{eqn_inproof_pp0}, we get 
\begin{align}
\label{eqn_inproof_p0m0_2}
  \Vert (p-p_0)(t) \Vert^2 + 2 \cLap \int_0^t \Vert (p-p_0)(s) \Vert^2_{H^1(\bE)} \ds 
  &\le \frac{\eps^2}{2\cLap \dm^2} \int_0^t \Vert \dot m(s)\Vert^2 \ds.
\end{align}
Note that we used that system~\eqref{eqn_op_B_eps} has a classical solution by Lemma~\ref{lem_clas_sol_B}. Further, we note that the derivative of the classical solution is again a mild solution of~\eqref{eqn_op_B_eps}. Hence, with an estimate of the form \eqref{eqn_estimate_clas} we obtain 
\begin{align*}
  \dm \int_0^t\Vert \dot m (s)\Vert^2 \ds
  &\le 2\dm \int_0^t\Vert \dot \mb (s)\Vert^2 \ds + 2\dm \int_0^t\Vert \dot \mt (s)\Vert^2 \ds  \\
  &\le 2\dm \Vert \mb \Vert^2_{H^1(0,T;L^2(\E))} + 2 \Big( \Vert \dot \pt(0) \Vert^2 + \eps\, \Vert \dot \mt(0)\Vert^2 + \widetilde C_\text{data} \Big)
\end{align*}
with a constant 
\[
  \widetilde C_\text{data} 
  = \widetilde C_\text{data} \big( \Vert \pb\Vert^2_{H^2(0,T;L^2(\E))}, \Vert \mb\Vert^2_{H^2(0,T;L^2(\E))}, \Vert \f\Vert^2_{H^1(0,T;L^2(\E))}, \Vert \g\Vert^2_{H^1(0,T;L^2(\E))}, \Vert \h\Vert^2_{H^2(0,T;\R^\Nin)} \big).
\]
It remains to bound the initial values of $\dot \pt$ and $\dot \mt$. For $\dot \pt(0)$ we use \eqref{eqn_inproof_ph_a} and the fact that we can bound $\calK^* \mt(0)$ by $\Vert \mstar \Vert_{L^2(\E)}$. For the estimate of $\dot \mt(0)$ we apply~\eqref{eqn_inproof_ph_b} and~\eqref{eqn_op_B_station_b} to obtain
\begin{align*}
 \eps\, \Vert \dot \mt(0)\Vert^2
  &\le 2\eps\, \Vert \dot {\mb}(0) \Vert^2 + \frac 2\eps\, \Vert \calK \pt(0) + d \mt(0) - \f(0) + \calK \Ce^- h(0) \Vert^2 \\
  &= 2\eps\, \Vert \dot {\mb}(0) \Vert^2 + \frac 2\eps\, \Vert \calK p(0) + d m(0) - f(0)  \Vert^2.
\end{align*}
Finally, we apply Corollary~\ref{cor_op_B_station} to bound $\Vert \pb\Vert_{H^2(0,T;L^2(\E))}$ and $\Vert \mb\Vert_{H^2(0,T;L^2(\E))}$ in terms of~$\Vert \r\Vert_{H^2(0,T;\R^\Nout)}$. 
\end{proof}
As mentioned above, we now consider system~\eqref{eqn_inproof_pp0} as a hyperbolic system. Thus, we need to assume the existence of $\dot m_0$, which then appears on the right-hand side. 
\begin{theorem}[First-order approximation II]
\label{thm_p0m0_2}
Suppose right-hand sides $\f \in H^1(0,T;L^2(\E))$, $\g \in W(0,T;L^2(\E),H^{-1}(\bE))$, $\h \in H^2(0,T; \R^\Nin)$, and $\r \in H^1(0,T; \R^\Nout)$. Further, assume consistent initial data $p(0)\in H^1(\bE)$ and $m(0) \in L^2(\E)$ and that there exists a function $\pstar\in L^2(\E)$ satisfying~\eqref{def_pstar}. Then, for $t\le T$ it holds that  
\begin{align*}
  \Vert (p-p_0)(t) \Vert^2 + \eps\, \Vert (m-m_0)(t) \Vert^2 &+ \dm \int_0^t \Vert (m-m_0)(s) \Vert^2 \ds\\
  &\qquad \le \frac{\eps}{\dm} \, \Vert \calK p(0) - f(0) + d m(0) \Vert^2 + \eps^2\,  C_\text{data}
\end{align*}
with a constant 
\[
  C_\text{data} 
  = C_\text{data} \big( \Vert \pstar \Vert^2_{L^2(\E)}, \Vert \f\Vert^2_{H^1(0,T;L^2(\E))}, \Vert \g\Vert^2_{W(0,T;L^2(\E),H^{-1}(\bE))}, \Vert \h\Vert^2_{H^2(0,T;\R^\Nin)}, \Vert \r\Vert^2_{H^1(0,T;\R^\Nout)} \big).
\]
\end{theorem}
\begin{proof}
Adding $-\eps\dot m_0$ to~\eqref{eqn_inproof_pp0_b}, we obtain the system 
\begin{alignat*}{5}
 \ddt (p-p_0)&\, +\, & a (p-p_0)&\,-\,& \calK^* (m-m_0) &= 0  &&\qquad \text{in } H^{-1}(\bE), \\
 \eps\, \ddt (m-m_0)&\, +\, &\calK (p-p_0)&\,+\,& d (m-m_0) &= -\eps\, \dot m_0 &&\qquad \text{in }[L^2(\E)]^*. 
\end{alignat*}
By Lemma~\ref{lem_weak_op_B_noeps_reg} the given assumptions guaranty the existence of $\dot m_0 \in L^2(0,T;L^2(\E))$. The result then follows by an estimate of the mild solution as in \eqref{eqn_estimate_clas}.
\end{proof}
\begin{remark}
\label{rem_firstorderineps}
In the case $\eps \dot m(0) = -\calK p(0) + f(0) - d m(0) = 0$, which is equivalent to $m(0) = m_0(0)$, Theorems~\ref{thm_p0m0} and~\ref{thm_p0m0_2} state that~$p_0$ is a first-order approximation of $p$ in terms of $\eps$. Further, $m_0$ is a first-order approximation of $m$ in $L^2(0,T;L^2(\E))$ and, under the conditions of Theorem~\ref{thm_p0m0_2}, an approximation of order $\frac 12$ in $C([0,T],L^2(\E))$.
\end{remark}
\begin{remark}
\label{rem_finiteinfinite}
The estimates of Theorems~\ref{thm_p0m0} and~\ref{thm_p0m0_2} are in line with \cite[Th.~1]{EggK17ppt}. For $m-m_0$ the obtained results also match with the finite-dimensional case analyzed in \cite[Ch.~2.5, Th.~5.1]{KokKO99}. For~$p-p_0$, however, one has $\|p-p_0\|_{L^\infty(0,T)} = \O(\eps)$ in the finite-dimensional case, independent of the initial data. We emphasize that the derived estimates are optimal in the infinite dimensional case as shown in Appendix~\ref{app_sharpness} and numerically confirmed in Section~\ref{sect:numerics}.
\end{remark}
\begin{remark}
Comparable estimates for the Lagrange multiplier, which only exists in a distributional sense under the given assumptions of Theorems~\ref{thm_p0m0} and~\ref{thm_p0m0_2}, use the inf-sup stability of $\Ce$ and read  
\begin{align*}
  \big| \int_0^t (\lambda - \lambda_0)(s) \ds \big|^2 
  \lesssim \| (p-p_0)(t)\|^2 +  t \int_0^t \| (m-m_0)(s)\|^2 \ds.
\end{align*}
Thus, the error in the primitives is of the same order as the pressure. 
\end{remark}
%
%
\subsection{Second-order approximation}\label{sect:expansion:second} 
In order to obtain a better approximation of $p$ and $m$, which we will exploit for the time discretization in Section~\ref{sect:timeInt}, we include the second term of the $\eps$-expansion~\eqref{eqn_epsExp}. The tuple $(p_1, m_1, \lambda_1)$ solves the system 
\begin{subequations}
\label{eqn_op_B_approx_1}
\begin{alignat}{5}
 \dot p_1 &\,+\, & a p_1 &\, - \, & \calK^* m_1 &\, +\, &\Ce^*\lambda_1 &= 0  &&\qquad \text{in }[H^1(\bE)]^*, \label{eqn_op_B_approx_1_a}\\
 & & \calK p_1 &\, +\, & d m_1 & & &= -\dot m_0  &&\qquad \text{in }[L^2(\E)]^*, \label{eqn_op_B_approx_1_b}\\
 & & \Ce p_1 &  & & & &= 0 &&\qquad \text{in } \R^{\Nin} \label{eqn_op_B_approx_1_c}
\end{alignat}
\end{subequations}
with the initial condition $p_1(0) = 0$. As usual, we first discuss the solvability of the system. Since \eqref{eqn_op_B_approx_1} can be written as a parabolic equation for $p_1$, we only need to analyze the regularity of the right-hand side, i.e., of $m_0$. The weak differentiability of $m_0$ has been discussed in Lemma~\ref{lem_weak_op_B_noeps_reg} such that an application of Lemma~\ref{lem_p0m0} directly leads to the following result.
\begin{lemma}[Existence of a weak solution $(p_1, m_1)$]\label{lem_p1m1}
Assume $\f\in H^1(0,T;L^2(\E))$, $\g\in W^{1,2}(0,T;L^2(\E), H^{-1}(\bE))$, $\h \in H^2(0,T; \R^{\Nin})$, and $\r \in H^1(0,T; \R^{\Nout})$. Further, let the initial data $p(0)\in H^1(\bE)$ be consistent and some $\pstar\in L^2(\E)$ satisfy~\eqref{def_pstar}. Then, system~\eqref{eqn_op_B_approx_1} is uniquely solvable with 
\[
  p_1 \in C([0,T], L^2(\E)) \cap L^2(0,T;H_0^1(\bE)), \qquad
  m_1 \in L^2(0,T;L^2(\E)).
\]
Furthermore, $\lambda_1$ exists in a distributional sense.
\end{lemma}
In the following, we are interested in the approximation property of $\hat p = p_0+\eps p_1$ and $\hat m = m_0+\eps m_1$. Obviously, this requires additional regularity assumptions. As noted in the previous subsection, the initial data may cause a reduction in the $\eps$-order of the approximation, regardless of the regularity of the data, cf.~Remark~\ref{rem_firstorderineps}. To focus on the improvements resulting from the incorporation of $p_1$ and $m_1$, we assume in the following that $\calK p(0) - f(0) + d m(0) = 0$, i.e., $m(0) = m_0(0)$. 
\begin{theorem}[Second-order approximation]
\label{thm_p1m1}
Consider right-hand sides $\f, \g \in H^2(0,T;L^2(\E))$, $\h \in H^3(0,T; \R^\Nin)$, and $\r \in H^3(0,T; \R^\Nout)$ with $\dot r(0) = 0$. Further, assume consistent initial data $p(0)\in H^1(\bE)$ and $m(0) \in L^2(\E)$ as well as the existence of $\mstar\in L^2(\E)$ satisfying \eqref{def_mstar} and of $\pstar\in H^1_0(\bE)$ satisfying \eqref{def_pstar}. Then, assuming $\calK p(0) - f(0) + d m(0) = 0$, we obtain for $t\le T$, 
\begin{align*}
  \Vert (p-\hat p)(t)\Vert^2 + \cLap \int_0^t \Vert (p-\hat p)(s) \Vert^2_{H^1(\bE)} \ds &+ \dm \int_0^t \Vert (m-\hat m)(s)\Vert^2\ds \\
  &\qquad\le \eps^3 \hat C\, \Vert \dot f(0) - \calK \dot p (0)\Vert^2 +  \eps^4 \hat C_\text{data} 
\end{align*}
with constants $\hat C = \hat C(\cLap, \dm)$ and 
\[
	\hat C_\text{data} 
	= \hat C_\text{data} \big( \Vert \pstar\Vert^2_{H^1(\bE)}, \Vert \f\Vert^2_{H^2(0,T;L^2(\E))}, \Vert \g\Vert^2_{H^2(0,T;L^2(\E))}, \Vert \h\Vert^2_{H^3(0,T;\R^\Nin)}, \Vert \r\Vert^2_{H^3(0,T;\R^\Nout)} \big).
\]
\end{theorem}
\begin{proof}
We consider the difference of the exact solution and $(\hat p, \hat m)$. Since $p-\hat p$ has vanishing boundary values, we can omit the Lagrange multiplier and obtain the system 
\begin{alignat*}{5}
 \ddt (p-\hat p) &\, + \,& a(p-\hat p) &\, -\, & \calK^* (m- \hat m) &= 0  &&\qquad \text{in }H^{-1}(\bE), \\
 & & \calK (p-\hat p)&\, +\, & d (m-\hat m) &= -\eps\, (\dot m - \dot m_0) &&\qquad \text{in }[L^2(\E)]^*
\end{alignat*}
with initial condition $(p-\hat p)(0) = 0$. Following the proof of Theorem~\ref{thm_p0m0}, we obtain
\begin{multline*} 
  \Vert (p-\hat p)(t)\Vert^2 + \cLap \int_0^t \Vert (p-\hat p)(s) \Vert^2_{H^1(\bE)} \ds  + \dm \int_0^t \Vert (m-\hat m)(s)\Vert^2 \ds\\
  \le  \frac{\eps^2}{4\dm^2}\Big(1+\frac{1}{\cLap \dm}\Big) \int_0^t \Vert (\dot m-\dot m_0)(s)\Vert^2 \ds. 
\end{multline*}
For an estimate of the integral on the right-hand side, we consider the formal derivative of system~\eqref{eqn_inproof_pp0}. Similar to equation~\eqref{eqn_inproof_p0m0}, we have
\begin{align}
\label{eqn_inproof_phat_estimate}
  \Vert (\dot p - \dot p_0)(t)\Vert^2 + \dm \int_0^t \Vert(\dot m-\dot m_0)(s)\Vert^2 \ds 
  \le \Vert (\dot p - \dot p_0)(0)\Vert^2 + \frac{\eps^2}{\dm} \int_0^t \Vert \ddot m (s)\Vert^2 \ds.
\end{align}
The assumption $\eps\, \dot m(0) = -\calK p(0) + f(0) - d m(0) = 0$ implies together with $p(0)=p_0(0)$ and equation~\eqref{eqn_inproof_pp0_b} that $m(0)=m_0(0)$. Inserting this in equation~\eqref{eqn_inproof_pp0_a}, we obtain $\dot p(0) = \dot p_0(0)$ in $L^2(\E)$ such that the first term vanishes. 
It remains to find an estimate of the integral of $\ddot m$. For this, we decompose $m$ into $m = \mb + \mt$ as in Theorem~\ref{thm_p0m0}. Note that $\mb$ and its derivatives only depend on the given data by Corollary~\ref{cor_op_B_station}. Using once more the formal derivative, we obtain by estimate~\eqref{eqn_estimate_clas} that 
\[
  \dm \int_0^t\big\Vert \ddot\mt (s)\big\Vert^2 \ds
  \le \big\Vert \ddot{\pt}(0) \big\Vert^2 + \eps\, \big\Vert \ddot\mt(0)\big\Vert^2 + \widetilde C_\text{data}.
\]
with  
\[
  \widetilde C_\text{data} 
  = \widetilde C_\text{data} \big( \Vert \f\Vert^2_{H^2(0,T;L^2(\E))}, \Vert \g\Vert^2_{H^2(0,T;L^2(\E))}, \Vert \h\Vert^2_{H^3(0,T;\R^\Nin)}, \Vert \r\Vert^2_{H^3(0,T;\R^\Nout)} \big).
\]
For the initial value of $\ddot{\pt}$ we note that by $\dot m(0) = 0$, $\dot r(0) = 0$, and equation~\eqref{eqn_op_B_station_a} it follows that $\calK^*\dot \mt(0)=0$ in $H^{-1}(\bE)$ and thus, using \eqref{eqn_inproof_ph_a}, 
\[
  \big\Vert \ddot\pt(0) \big\Vert^2
  = \big\Vert \dot\g(0) -a \dot p(0)- \ddot \pb(0) - \Ce^-\ddot\h(0) \big\Vert^2.
\]
For an estimate of $\ddot\mt(0)$ we note that   
\begin{align*}
  \eps^5\, \big\Vert \ddot\mt(0)\big\Vert^2
  \le 2\eps^5\, \big\Vert \ddot m(0)\big\Vert^2 + 2\eps^5\, \big\Vert \ddot\mb(0)\big\Vert^2
  = 2\eps^3\, \big\Vert \dot f(0) - \calK \dot p(0)\big\Vert^2 + 2\eps^5\, \big\Vert \ddot\mb(0)\big\Vert^2, 
\end{align*}
which gives the claimed estimate. We emphasize that the first term on the right-hand side is bounded in terms of the data, since 
\[
  \big\Vert \dot f(0) - \calK \dot p(0)\big\Vert^2
  \le 2\, \big\Vert \dot\f(0)\big\Vert^2 + 2\, \big\Vert \calK \dot p(0)\big\Vert^2.  
\]
For this, we use $\Vert \dot\f(0)\Vert \lesssim \Vert f \Vert_{H^2(0,T;L^2(\E))}$ and $\dot p(0) = \dot p_0 (0) = \dot \pt_0 (0) + \Ce^- \dot{h}(0)$, where $\dot \pt_0 (0)$ is bounded in $H^1_0(\bE)$ by equation~\eqref{eqn_op_B_noEps_pot} and $\pstar\in H^1_0(\bE)$. 
\end{proof}
\begin{remark}
Under the additional assumption $0=\dot{f}(0) - \calK \dot{p}(0) =\eps \ddot m(0)$, which is equivalent to $m(0)=\hat m(0)$ and $\dot{m}(0)=\dot{m}_0(0)$, 
Theorem~\ref{thm_p1m1} states that~$\hat p$ is a second-order approximation of $p$ in terms of $\eps$. Further, $\hat m$ is a second-order approximation of $m$ in $L^2(0,T;L^2(\E))$. 
\end{remark}
\begin{remark}
\label{rem_finiteinfinite2}
The estimate for $m-\hat m$ in Theorem~\ref{thm_p1m1} is again in line with the finite dimensional case, whereas the difference $p-\hat p$ is one order smaller, cf~\cite[Ch.~2.5, Th.~5.2]{KokKO99}. Again, the given estimates are sharp for the infinite dimensional case and numerically validated in Section~\ref{sect:numerics}. 
\end{remark}
Given the assumptions of Theorem~\ref{thm_p1m1} there exist regular solutions $\lambda$ and $\lambda_0$. We close this section with a remark on the approximation property of the Lagrange multiplier. 
\begin{remark}\label{rem:est_lemma}
Under the the assumptions from Theorem~\ref{thm_p1m1}, 
it follows that $\lambda_0$ is a first-order approximation of $\lambda$, i.e., 
\begin{align*}
\int_0^t |(\lambda - \lambda_0)(s)|^2 \ds 
&\lesssim \eps \hat C\, \Vert \dot f(0) - \calK \dot p (0)\Vert^2 +  \eps^2 \max \{ C_\text{data}, \hat C_\text{data} \}
\end{align*}
with the constants~$\hat C$, $C_{\text{data}}$, and~$\hat C_\text{data}$  from Theorems~\ref{thm_p0m0} and~\ref{thm_p1m1}. The associated proof
uses that the parabolic PDE~\eqref{eqn_op_B_noEps}, which is satisfied by the differences $p-p_0$, $m-m_0$, and $\lambda-\lambda_0$, fulfills the assumptions of Lemma~\ref{lem_weak_op_B_noeps_reg}. The order of approximation for $\lambda$ and $\lambda_0$ then follows by the estimates of the integrals of $\|\dot{m}\|^2$ and~$\|\ddot{m}\|^2$ as in the proof of Theorem~\ref{thm_p0m0} and~\ref{thm_p1m1}, respectively. 
\end{remark}

%% file: semidiscrete.tex
\section{Index of the Semi-discrete Systems}\label{sect:index}
Since the considered system~\eqref{eqn_op_B_eps} is a PDAE, a semi-discretization in space leads to a DAE. This section is devoted to the computation of the {\em differentiation index} of the resulting system. Roughly speaking, this index measures the minimal number of differentiation steps in order to extract an ODE, cf.~\cite[Def.~2.2.2]{BreCP96} for a precise definition and \cite{Meh13} for further index concepts. We emphasize that we do not consider the index of the PDAE.

A spatial discretization, e.g.~by finite elements, turns the linear operators into matrices. At this point, we gather a list of general assumptions on these matrices, rather than prescribing a precise discretization scheme. For the discretization of $m$ and $p$ we consider symmetric and positive definite mass matrices $M_1 \in \R^{n_m,n_m}$ and $M_2 \in \R^{n_p,n_p}$. The mass matrices including the space-dependent damping terms $d$, $a$ are denoted by $\Md \in \R^{n_m,n_m}$, $\Ma \in \R^{n_p,n_p}$, respectively. Both matrices are symmetric and semi-positive definite. Further, the discrete version of the linear operator $\calK$ is denoted by $K \in \R^{n_m,n_p}$.

As discrete versions of the constraint operators $\B$ and $\Ci$ from Section~\ref{sect:prelim:constraints} we introduce the matrices $B \in \R^{\Nin,n_p}$, which we assume to be of full row rank, and $C \in \R^{\Nout,n_p}$. We emphasize that the discretization of the dual operators equal the transpose of the corresponding matrices. 
Note that the data for the constraints, i.e., the Dirichlet data $\h$ and $\r$, are already finite-dimensional and need no further discretization. For the discretization of $\f$ and $\g$ the same notation as for the original right-hand sides. This means that we assume in this section that $\f\colon [0,T] \to \R^{n_m}$ and $\g\colon [0,T] \to \R^{n_p}$. Finally, we define 
\[
	\tilde \g := \g - C^T \r.
\]
Using for the semi-discrete variables the same notation as for the continuous ones, we obtain as spatial discretization of \eqref{eqn_op_B_eps} the system 
\begin{align}
\label{eqn_index2_DAE}
  \begin{bmatrix}
  \eps M_1  &   &   \\ 
  &  M_2 & \\
  &    &  0
  \end{bmatrix}
  \begin{bmatrix} \dot m \\ \dot p \\ \dot \lambda \end{bmatrix}
  = \begin{bmatrix}
  - \Md & -K &  \\
  K^T &  - \Ma & -B^T  \\   
   & -B & 
  \end{bmatrix}
  \begin{bmatrix} m \\ p \\ \lambda \end{bmatrix}
  + \begin{bmatrix} \f \\ \tilde \g \\ \h \end{bmatrix}.
\end{align}
This DAE has index 2. To see this, we consider the derivative of the constraint, which leads to the system 
\[
  \begin{bmatrix}
  \eps M_1  &   &   \\ 
  &  M_2 & B^T\\
  &  B  &  
  \end{bmatrix}
  \begin{bmatrix} \dot m \\ \dot p \\ \lambda \end{bmatrix}
  = \begin{bmatrix}
  -\Md & -K &  \\
  K^T & - \Ma  &  \\   
  &  &  0 
  \end{bmatrix}
  \begin{bmatrix} m \\ p \\ \lambda \end{bmatrix}
  + \begin{bmatrix} \f \\ \tilde \g \\ \dot{\h} \end{bmatrix}.
\]
Note that the vector including the variables on the left-hand side contains the derivatives of $m$ and $p$ but the Lagrange multiplier $\lambda$ without any derivative. Since the matrix~$B$ is assumed to have full row rank and $M_2$ is positive definite, the matrix on the left-hand side is invertible for $\eps>0$. A multiplication from the left by the inverse then yields differential equations for $m$ and $p$ together with an algebraic equation for $\lambda$. Since one differentiation step was sufficient to extract these equations, the original DAE was of index 2.

Next, we discuss the spatial discretization of the parabolic limit case~\eqref{eqn_op_B_noEps}. This means nothing else then setting $\eps=0$ in equation~\eqref{eqn_index2_DAE}. Thus, we have the DAE 
\[
  \begin{bmatrix}
  0  &   &   \\ 
  &  M_2 & \\
  &    &  0
  \end{bmatrix}
  \begin{bmatrix} \dot m \\ \dot p \\ \dot \lambda \end{bmatrix}
  = \begin{bmatrix}
  - \Md & -K &  \\
  K^T & - \Ma  & -B^T  \\   
  & -B & 
  \end{bmatrix}
  \begin{bmatrix} m \\ p \\ \lambda \end{bmatrix}
  + \begin{bmatrix} \f \\ \tilde \g \\ \h \end{bmatrix}.
\]
Since the damping parameter $d$ is positive, $\Md$ is invertible and the first equation allows to write $m$ in terms of~$p$. It remains a system of the form 
\[ 
  M_2\dot p = K^T m(p) - \Ma p - B^T \lambda + \tilde \g, \qquad
  B p  = \h
\]
with the typical semi-explicit index-2 structure, cf.~\cite[Ch.~VII.1]{HaiW96}.
\begin{remark}
Another possible discretization is discussed in~\cite{HucT17} and takes the topology of the network into account leading to an index-1 DAE. Therein, the semi-discrete variables only include the values of $p$ and $m$ at the endpoints of an edge. 
\end{remark}
\begin{remark}
We discuss the {\em port-Hamiltonian structure} of the system equations, cf.~\cite{Van13}. For this, we consider the right-hand sides as inputs of the system. Following the DAE formulation introduced in~\cite[Def.~4]{BeaMXZ17ppt}, we may write~\eqref{eqn_index2_DAE} in the form 
\[
	\begin{bmatrix}
	\eps M_1  &  & \\ 
	&  M_2 & \\
	&  &  0
	\end{bmatrix}
	\begin{bmatrix} \dot m \\ \dot p \\ \dot \lambda \end{bmatrix}
	= \Bigg( 
	\begin{bmatrix} 0 & -K &  \\  K^T & 0 & -B^T \\ & B & 0 \end{bmatrix}
	- \begin{bmatrix}  \Md & & \\  &  \Ma &  \\ & & 0  \end{bmatrix} 
	\Bigg) 
	\begin{bmatrix} m \\ p \\ \lambda \end{bmatrix}
	+ \begin{bmatrix} \f \\ \tilde \g \\ \h \end{bmatrix}
\]
with the sum of a skew-symmetric structure matrix (describing the energy flux) and a symmetric dissipation matrix on the right-hand side. 
%
%
The corresponding energy (also called {\em Hamiltonian}) is given by 
\[
  H(m, p) 
  = \tfrac 12 \eps\, m^T M_1 m + \tfrac 12 p^T M_2 p.  
\]
Note that this depends on the parameter $\eps$. Thus, the system maintains its port-Hamiltonian structure for $\eps=0$ but with a different energy. 
We emphasize that the port-Hamiltonian structure is already given in the continuous system~\eqref{eqn_op_B_eps}, cf.~\cite{JacZ12}.
\end{remark}
%
%
Finally, we turn to the system for the second-order terms of the $\eps$-expansion, i.e., system~\eqref{eqn_op_B_approx_1} with solution $(p_1,m_1,\lambda_1)$. More precisely, we consider the through $\dot m_0$ coupled system of~\eqref{eqn_op_B_noEps} and~\eqref{eqn_op_B_approx_1}. As a consequence, the overall system has no saddle point structure anymore. We take once more the same notion for the continuous and semi-discrete variables and introduce 
\[
  x 
  := \big[\, m^T_0,\ p^T_0,\ p^T_1,\ m^T_1,\ \lambda^T_0,\ \lambda^T_1\, \big]^T.
\]
The spatial discretization of the coupled system then leads to the DAE
\begin{align}
\label{eqn_semidiscr_coupled}
%
  \mathbf{M}\, 
  \dot x   
  = 
  \left[
  \begin{array}{ccc|ccc}  
  & & -K & -\Md\\
  K^T & - \Ma  &  & & -B^T  \\
  &  & - \Ma  & K^T & & -B^T  \\  \hline
  -\Md &  -K &  \\  
  & -B & \\
  & & -B & 
  \end{array}
  \right]
  x
  + \begin{bmatrix} 0 \\  \tilde g \\ 0  \\ f \\ h \\ 0 \end{bmatrix}
\end{align}
with the block-diagonal mass matrix $\mathbf{M} := \diag \{M_1, M_2, M_2, 0, 0, 0\}$. This DAE is again of index 2. To see this, we follow~\cite[p.~456]{HaiW96} and note that the matrix
\begin{align*}
&\begin{bmatrix}   
  -\Md &  -K &  \\  
   & -B & \\
   & & -B  
  \end{bmatrix} 
  \begin{bmatrix}
  M_1^{-1} & & \\
  & M_2^{-1} &  \\
  & & M_2^{-1} 
  \end{bmatrix} 
  \begin{bmatrix}
  -\Md & & \\
  & -B^T &  \\
  K^T & & -B^T 
  \end{bmatrix} 
\\
 &\qquad\qquad=\ \begin{bmatrix}
  \Md M_1^{-1} \Md & KM_2^{-1}B^T & \\
  & BM_2^{-1}B^T  \\
  -BM_2^{-1}K^T & & BM_2^{-1}B^T 
  \end{bmatrix}
\end{align*}
is invertible, due to the full rank property of $B$.  

%% file: timeintegration.tex
\section{Temporal Discretization}\label{sect:timeInt}
In this section, we investigate the behavior of the first- and second-order approximations $p_0, m_0, \lambda_0$ and $\hat p, \hat m, \hat \lambda$ under a discretization in time. We consider Runge-Kutta methods and focus, in particular, on the implicit Euler scheme. The convergence of such schemes for semi-explicit PDAEs of parabolic type was analyzed in~\cite{AltZ18} and is thus applicable for the case with $\eps=0$, i.e., for system~\eqref{eqn_op_B_noEps}. The combination with the approximation results from Section~\ref{sect:expansion} then yields a discretization scheme for the full problem~\eqref{eqn_op_B_eps}. 

Unfortunately, the assumed structure in~\cite{AltZ18} is too restrictive for the analysis of the coupled system consisting of~\eqref{eqn_op_B_noEps} and~\eqref{eqn_op_B_approx_1} and thus, needs an extension. Another major difficulty is the appearance of the variable~$m_0$, which we discuss in Section~\ref{sect:timeInt:second}. 
%
%
\subsection{First-order approximation}\label{sect:timeInt:first}
Based on the results of Section~\ref{sect:expansion}, we propose to approximate $p$ and $m$ by a temporal discretization of~$p_0$ and $m_0$, respectively. Due to the PDAE structure of~\eqref{eqn_op_B_noEps}, we have 'differential' and 'algebraic' equations, which causes difficulties within the time discretization. One particular reason is the high sensitivity to perturbations of the right-hand side, cf.~\cite{Alt15}. In the finite-dimensional case it is well-known that this may decrease the convergence order, see~\cite[Ch.~6]{HaiW96} and~\cite[Ch.~5.2]{KunM06}. 
To bypass these issues, we apply a {\em regularization} to the system involving the so-called {\em hidden constraint}, i.e., the derivative of the constraint. The regularization approach considered in~\cite{AltZ18} introduces an additional Lagrange multiplier~$\mu_0\colon [0,T] \to \R^\Nin$ and adds the hidden constraint explicitly to the system equations. For~\eqref{eqn_op_B_noEps} this procedure results in 
\begin{subequations}
\label{eqn_op_B_noEps_index1}
\begin{alignat}{6}
  \dot p_0 &\,+\, & a p_0 &\,-\ & \calK^* m_0 &\,+\, & \Ce^*\lambda_0 &\,+\, & \Ce^*\mu_0 &= \g-\Ci^*\r  &&\qquad\text{in }[H^1(\bE)]^*, \label{eqn_op_B_noEps_index1_a}\\
  & & \calK p_0 &\,+\, & d m_0 & & & & &= \f  &&\qquad\text{in }[L^2(\E)]^*, \label{eqn_op_B_noEps_index1_b}\\
  & & \Ce p_0 & & & & &\,-\,& M \mu_0 &= \h &&\qquad\text{in } \R^{\Nin}, \label{eqn_op_B_noEps_index1_c}\\
  \Ce \dot{p}_0 & & & & & & & & &= \dot{\h} &&\qquad\text{in } \R^{\Nin} \label{eqn_op_B_noEps_index1_d}
\end{alignat}
\end{subequations}
with an arbitrary invertible matrix $M\in \R^{\Nin \times \Nin}$. Under certain regularity assumptions on~$p_0$, every solution of~\eqref{eqn_op_B_noEps_index1} with a consistent initial condition implies $\mu_0 =0$ and that $(p_0,m_0,\lambda_0)$ solves~\eqref{eqn_op_B_noEps}, cf.~\cite[Lem~3.6]{AltZ18}. Note that the hidden constraint for~$p_0$ is stated explicitly in~\eqref{eqn_op_B_noEps_index1_d}.
\begin{remark}
Another regularization approach was introduced in~\cite{AltH15} and uses dummy variables. For this, let $\P_\text{c}$ be an arbitrary complement of $H^1_0(\bE)$ in $H^1(\bE)$ leading to the unique decomposition $p_0 = \pt_0 + p_{0,\text{c}}$ with $\pt_0 \in H^1_0(\bE)$ and $p_{0,\text{c}} \in \P_\text{c}$. 
Then, the regularization of~\eqref{eqn_op_B_noEps} is given by 
\begin{subequations}
\label{eqn_op_B_noEps_index1_AltH}
\begin{alignat}{6}
  \dot{\pt}_0 &\,+\, & q_{0,\text{c}} &\,+\, & a(\pt_0 + p_{0,\text{c}}) &\,-\, & \calK^* m_0 &\,+\, & \Ce^*\lambda_0 &= \g-\Ci^*\r  &&\qquad\text{in }[H^1(\bE)]^*, \label{eqn_op_B_noEps_index1_AltH_a}\\
  & & & & \calK (\pt_0 + p_{0,\text{c}}) &\,+\, & d m_0 & & &= \f  &&\qquad\text{in }[L^2(\E)]^*, \label{eqn_op_B_noEps_index1_AltH_b}\\
  & & & & \Ce p_{0,\text{c}}\  & & & & &= \h &&\qquad\text{in } \R^{\Nin}, \label{eqn_op_B_noEps_index1_AltH_c}\\
  & & \Ce q_{0,\text{c}} & & & & & & &= \dot{\h} &&\qquad\text{in } \R^{\Nin}  \label{eqn_op_B_noEps_index1_AltH_d}
\end{alignat}
\end{subequations}
including the dummy variable $q_{0,\text{c}} \in \P_\text{c}$, which approximates the derivative of $p_{0,\text{c}}$. By \cite[Th.~2.3]{AltH15} and \cite[Lem.~2.5]{AltZ18} every solution of~\eqref{eqn_op_B_noEps_index1_AltH} generates a solution of~\eqref{eqn_op_B_noEps} by the triple $(\pt_0 + p_{0,\text{c}}, m_0,\lambda_0)$. 
\end{remark}
%
%
\subsubsection{Implicit Euler scheme}\label{sect:timeInt:first:Euler}
We first investigate the approximation obtained by the implicit Euler scheme. Runge-Kutta schemes are then discussed in Section~\ref{sect:timeInt:first:RK}. For the temporal discretization of~\eqref{eqn_op_B_noEps_index1} we consider a uniform partition of the interval $[0,T]$ with step size $\tau=T/n$, $n\in\N$, and time steps $t_j=\tau j$, $j=0,\ldots,n$. The time-discrete system then reads 
\begin{subequations}
\label{eqn_disc_p0_reg}
\begin{alignat}{6}
  \Dt \pj &\,+\, & a\pj &\,-\, & \calK^* \mj &\,+\, & \Ce^*\lj &\,+\, & \Ce^\ast \muj &= \g_j-\Ci^*\r_j  &&\qquad\text{in }[H^1(\bE)]^*, \label{eqn_disc_p0_reg_a}\\
  & & \calK \pj &\,+\, & d \mj & & & & &= \f_j  &&\qquad\text{in }[L^2(\E)]^*, \label{eqn_disc_p0_reg_b}\\
  & & \Ce \pj & & & & &\,-\, & M \muj &= \h_j &&\qquad\text{in }\R^{\Nin}, \label{eqn_disc_p0_reg_c}\\
  \Ce \Dt \pj & & & & & & & & &= \dot{\h}_j &&\qquad\text{in }\R^{\Nin}. \label{eqn_disc_p0_reg_d}
\end{alignat}
\end{subequations}
Therein, $\Dt$ denotes the discrete derivative, i.e., $\Dt\pj := (\pj - \pjm)/\tau$, and $p_{0,0}:=p(0)$. The right-hand sides of~\eqref{eqn_disc_p0_reg} should be appropriate approximations of the right-hand sides of~\eqref{eqn_op_B_noEps_index1} at time~$t_j$. However, the given smoothness may not always allow a point evaluation. 
We introduce the piecewise constant function $\g_\tau\colon [0,T] \to [H^1(\bE)]^\ast$ with $\g_\tau(t):= g_j$ for $t\in (t_{j-1},t_j]$, $j=1,\ldots,n$. Analogously, we define $\f_\tau$, $\h_\tau$, $\dot{\h}_\tau$, and $\r_\tau$. For the convergence analysis we assume that for $\tau\to 0$ we have 
\begin{equation}
\label{ass_disc_righthand_side}
\begin{aligned}
\g_\tau &\to \g \text{ in } L^2(0,T;[H^1(\bE)]^\ast), & \f_\tau &\to \f \text{ in } L^2(0,T;L^2(\E)), & \r_\tau &\to \r \text{ in } L^2(0,T;\R^{\Nout}),\\
\h_\tau &\to \h \text{ in } L^\infty(0,T;\R^{\Nin}), & \dot{\h}_\tau &\to \dot{\h} \text{ in } L^2(0,T;\R^{\Nin}).
\end{aligned}
\end{equation}
\begin{remark}
\label{rem:Dh_doth}
The function~$\dot{\h}_\tau$ denotes the approximation of~$\dot{\h}$ and not the derivative of~$\h_\tau$. 
We emphasize that defining $\dot \h_j$ by the discrete derivative $\Dt \h_j$ makes the performed regularization useless, since the two constraints would be redundant. In fact, the regularization aims to include information on the derivative $\dot{\h}$ to the system rather than using the discrete derivative. 
\end{remark}
\begin{lemma}\label{lem_disc_qualti}
Let the conditions of Lemma~\ref{lem_p0m0} be satisfied and $(p_0,m_0,\lambda_0)$ the solution of the PDAE~\eqref{eqn_op_B_noEps}. Suppose that $\g_\tau$, $\f_\tau$, $\h_\tau$, $\dot{\h}_\tau$, and $\r_\tau$ converge as stated in~\eqref{ass_disc_righthand_side}. Then, system~\eqref{eqn_disc_p0_reg} has a unique solution for every $j=1,\ldots,n$. Further, let $\potC$, $\motC$, and $\muotC$ denote the piecewise constant functions and~$\potL$ the continuous and piecewise linear function defined via~$\pj$, $\mj$, and $\muj$, respectively. Then, as $\tau \to 0$ we have 
\begin{align*}
\potC &\to p_0 &&\text{in } L^2(0,T;H^1(\bE)), & \potL &\to p_0 &&\text{in } L^2(0,T;L^2(\E)),\\
\motC &\to m_0 &&\text{in } L^2(0,T;L^2(\E)), & \muotC &\to 0 &&\text{in } L^\infty(0,T;\R^{\Nin}).
\end{align*}
\end{lemma}
\begin{proof}
By inserting equation~\eqref{eqn_op_B_noEps_index1_b} into~\eqref{eqn_op_B_noEps_index1_a} and analogously equation~\eqref{eqn_disc_p0_reg_b} into~\eqref{eqn_disc_p0_reg_a}, the statements follow immediately by \cite[Lem.~4.2 \& Th.~4.3]{AltZ18}.
\end{proof}
Lemma~\ref{lem_disc_qualti} provides a qualitative statement about the convergence of the piecewise constant and linear approximations. For a quantitative result we need additional regularity as for Lemma~\ref{lem_weak_op_B_noeps_reg}. Under these assumptions, one can use function evaluations to define the right-hand sides in~\eqref{eqn_disc_p0_reg} which then automatically satisfy~\eqref{ass_disc_righthand_side}.
\begin{lemma}\label{lem_disc_parabolic_estimate}
Let the assumptions from Lemma~\ref{lem_weak_op_B_noeps_reg} be satisfied and suppose that the right-hand sides of~\eqref{eqn_disc_p0_reg} are defined via function evaluations at time~$t_j$. With a generic constant $C_\text{data}$ only depending on the data, it then holds that
\begin{align*}
  \|\potC  - p_0\|_{L^2(0,T;H^1(\bE))} &\leq \tau\, C_\text{data}, 
  & \|\potL - p_0\|_{C(0,T;L^2(\E))} &\leq \tau\, C_\text{data},\\
  \|\motC - m_0  \|_{L^2(0,T;L^2(\E))}&\leq \tau\, C_\text{data}, 
  & \|\muotC\|_{L^\infty(0,T;\R^{\Nin})} &\leq \tau\, C_\text{data}.
\end{align*}
\end{lemma}
\begin{proof}
Recall the definition of~$\pot$ introduced in the proof of Lemma~\ref{lem_weak_op_B_noeps_reg}. Similar to~\cite[Lem.~2.4]{AltZ18}, we consider the decomposition $H^1(\bE)=H^1_0(\bE) \oplus H^1_c$ with  
\[ 
  H^1_c := \{ p \in H^1(\bE)\,|\, \langle(\wlap + a) p,q\rangle =0 \text{ for all } q \in H^1_0(\bE)\},
\]
including the operator $\wlap$ from Lemma~\ref{lem_wlap}. By the discrete constraints~\eqref{eqn_disc_p0_reg_c} and~\eqref{eqn_disc_p0_reg_d} we decompose $\pj = \pjt + \Ce^-(\h_j - M \muj)$ and $\Dt\pj = \Dt\pjt + \Ce^- \dot \h_j$, where $\Ce^-$ denotes the right-inverse of $\Ce$ with range $H^1_c$, cf.~\cite[Lem.~2.5]{AltZ18}. Then, inserting equation~\eqref{eqn_op_B_noEps_index1_b} into~\eqref{eqn_op_B_noEps_index1_a} as well as~\eqref{eqn_disc_p0_reg_b} into~\eqref{eqn_disc_p0_reg_a}, we obtain 
\begin{align*}
  \Dt\pjt + (\wlap + a)\, \pjt 
  &= \g(t_j) - \Ci^\ast \r(t_j) + \calK^\ast d^{-1} f(t_j) - \Ce^- \dot \h(t_j) \\
  &= \dot{\pt}(t_j) + (\wlap + a)\, \pt(t_j) \qquad \text{in } H^{-1}(\bE).
\end{align*}
We now insert~$\tau\,(\pjt - \pot(t_j))\in H^1_0(\bE)$ as a test function and consider a Taylor expansion of~$\pot(t_{j-1})$. This then leads to the estimate 
\begin{equation}
\label{eqn_disc_diff}
\begin{aligned}
  &\| \pjt - \pot(t_j)\|^2 + \tau \cLap\, \| \pjt - \pot(t_j)\|_{H^1(\bE)}^2\\
  \leq\ & \| \pjmt - \pot(t_{j-1})\|^2 + \frac{1}{\tau \cLap} \big\| \int_{t_{j-1}}^{t_j} (s-t_{j-1})\, \ddot{\pot}(s) \ds\, \big\|^2_{H^{-1}(\bE)}\\
  \leq\ & \| \pjmt - \pot(t_{j-1})\|^2 + \frac{\tau^2}{3\cLap}  \int_{t_{j-1}}^{t_j} \|\ddot{\pot}(s)\|^2_{H^{-1}(\bE)} \ds.
\end{aligned}
\end{equation}
Since $\mu_{0,0}=0$, we have $\pt_{0,0} - \pot(0)=0$ and therefore by summing up all the estimates~\eqref{eqn_disc_diff} from $k=1,\ldots,j$, it follows that 
\begin{equation}\label{eqn_disc_pointwise}
  \| \pjt - \pot(t_j)\|^2 + \cLap \sum_{k=1}^j \tau\, \| \pkt - \pot(t_k)\|_{H^1(\bE)}^2 
  \leq \frac{\tau^2}{3\cLap}  \int_{0}^{t_j} \|\ddot{\pot}(s)\|^2_{H^{-1}(\bE)} \ds.
\end{equation}
For the stated approximation orders we also need an estimate of the Lagrange multiplier~$\mu$. A multiple application of equation~\eqref{eqn_disc_p0_reg_d} and Taylor's theorem yield together  
\begin{align}
|M\muj|^2 
&=\Big|\sum_{k=1}^j \tau \dot \h(t_k) - [\h(t_k) - \h(t_{k-1})]  \Big|^2 \notag\\
&= \Big|\sum_{k=1}^j \int_{t_{k-1}}^{t_k} (s-t_{k-1}) \ddot \h(s) \ds \Big|^2 
\leq \frac{\tau^3}{3} j \int_{0}^{t_j} |\ddot \h(s)|^2 \ds \leq \tau^2 \frac{T}{3} \int_{0}^{T} |\ddot \h(s)|^2 \ds. \label{eqn_est_disc_mu}
\end{align}
Therefore, the approximation order for $\|\muotC\|^2_{L^\infty(0,T;\R^{\Nin})}= \max_{j=1,\ldots,n} |\muj|^2$ follows by the invertibility of~$M$. 
%
%
%
The estimate for the piecewise constant approximation of~$p_0$ follows by
\begin{align*}
\int_{t_{k-1}}^{t_k} \| \pot(t_k) - \pot(s)\|^2_{H^1(\bE)} \ds &= \int_{t_{k-1}}^{t_k} \big\| \int_{s}^{t_k} \dot{\pot}(\eta)\deta\, \big\|^2_{H^1(\bE)} \ds 
\leq \frac{\tau^2}{2} \int_{t_{j-1}}^{t_j}  \|  \dot{\pot}(s) \|^2_{H^1(\bE)} \ds
\end{align*}
together with estimates~\eqref{eqn_disc_pointwise}-\eqref{eqn_est_disc_mu} and the triangle inequality. Recall that $\dot{\pot}(s)$, as well as its derivative, can be bounded in terms of the data. 
%
%
The convergence rate of~$\motC$ then follows if we consider the difference of the equations~\eqref{eqn_op_B_noEps_index1_b} and~\eqref{eqn_disc_p0_reg_b}. 

Finally, we discuss the error $\potL - p_0$. For this, note that a Taylor expansion implies for every $t\in [t_{j-1},t_j]$ that 
\begin{align*}
  &\phantom{2\,} | h(t_{j-1}) - \dot h(t_j) (t-t_{j-1}) - h(t) |^2\\
  \leq\ & 2\,| h(t_{j}) - \dot h(t_j) (t-t_{j}) - h(t) |^2 + 2\,| h(t_{j-1}) - h(t_j) - \dot h(t_j) (t_j-t_{j-1})  |^2\\
  =\ & 2\, \big| \int^{t_{j}}_t (s-t)\, \ddot h (s) \ds \big|^2 + 2\, \big| \int^{t_{j}}_{t_{j-1}} (s-t_{j-1})\, \ddot h (s) \ds\, \big|^2 \\
  \leq\ & \frac{4 \tau^3}{3} \int_{t_{j-1}}^{t_j} |\ddot h (s) |^2 \ds.
\end{align*}
Further, with $\chi_{[a,b]}$ denoting the indicator function for an interval~$[a,b]$ we have
\begin{align*}
  &\big\| \pot(t_{j-1}) + \frac{\pot(t_{j}) - \pot(t_{j-1})}{\tau}(t-t_{j-1}) - \pot(t) \big\|^2_{L^2(\E)} \\
  =\, & \Big\langle \int_{t_{j-1}}^{t_j} \Big( \tfrac{t-t_{j-1}}{\tau}(t_j-s) - \chi_{[t_{j-1},t]}(t-s)\Big) \ddot \pot(s)\ds , \int_{t_{j-1}}^{t_j} \Big(\tfrac{t-t_{j-1}}{\tau} - \chi_{[t_{j-1},t]} \Big) \dot \pot(s)\ds\Big\rangle\\
  \leq\, & \Big(\frac{\tau^{3}}{48} \int_{t_{j-1}}^{t_j}\|\ddot \pot(s)\|^2_{H^{-1}(\bE)} \ds \Big)^{1/2} \cdot \Big(\frac{\tau}{4} \int_{t_{j-1}}^{t_j} \|\dot \pot(s)\|^2_{H^{1}_0(\bE)}\ds\Big)^{1/2}\\
  \leq\, &\frac{\tau^2}{\sqrt{768}} \int_{t_{j-1}}^{t_j}\|\ddot \pot(s)\|^2_{H^{-1}(\bE)} + \|\dot \pot(s)\|^2_{H^{1}_0(\bE)}\ds.
\end{align*}
A combination of the latter two estimates with ~\eqref{eqn_disc_pointwise},~\eqref{eqn_est_disc_mu}, and the triangle inequality then leads to the stated estimate for~$\potL- p_0$.
\end{proof}
\begin{theorem}\label{thm_euler_p0m0}
Suppose a uniform partition in time with step size~$\tau$, $m(0) = m_0(0)$, and that the assumptions of Theorem~\ref{thm_p0m0} and Lemma~\ref{lem_disc_parabolic_estimate} are satisfied. 
Then, the piecewise constant functions~$\potC$,~$\motC$ and the piecewise linear function~$\potL$ defined through the implicit Euler scheme~\eqref{eqn_disc_p0_reg} serve as an approximation of the solution to~\eqref{eqn_op_B_eps} in the sense that 
\begin{align*}
  \Vert p - \potC \Vert_{L^2(0,T;H^1(\bE))}
  + \Vert m - \motC \Vert_{L^2(0,T;L^2(\E))}
  + \Vert p - \potL \Vert_{C(0,T;L^2(\E))}
  \le (\eps+\tau)\, C_\text{data}.
\end{align*}
Therein, $C_\text{data}$ denotes once more a constant only depending on the data. 
\end{theorem}
\begin{proof}
The result follows directly by the combination of Theorem~\ref{thm_p0m0} with Lemma~\ref{lem_disc_parabolic_estimate}. For this, we apply the triangle inequality and obtain
\[
  \Vert (p - \potL)(t) \Vert
  \le \Vert (p - p_{0})(t) \Vert + \Vert (p_0 - \potL)(t) \Vert
  \le (\eps+\tau)\, C_\text{data}.
\]
The estimates for the piecewise constant approximations follow in the same way. 
%
\end{proof}
%
%
\subsubsection{Runge-Kutta methods}\label{sect:timeInt:first:RK}
A Runge-Kutta scheme is based on a {\em Butcher tableau} 
\[ 
  \begin{array}{c|c} \fc & \fA \\ \hline & \fb^T \end{array} 
\]
with~$\fb,\fc\in \R^s$ and~$\fA\in \R^{s,s}$. As in~\cite{AltZ18} we assume that the resulting method is algebraically stable, that~$\fA$ is invertible, and $\fb^T \fA^{-1} \ones =1$ with~$\ones= [1, \ldots, 1]^T \in \R^s$. The temporal discretization of~\eqref{eqn_op_B_noEps_index1} then leads to the approximations 
\begin{equation}
\label{eqn_disc_p0_reg_RK_step}
  \pj = \fb^T \fA^{-1} \psj, \ \
  \mj = \fb^T \fA^{-1} \msj, \ \
  \lj = \fb^T \fA^{-1} \lsj, \ \
  \muj = \fb^T \fA^{-1} \musj,
\end{equation}
where $\psj \in H^1(\bE)^s$, $\msj \in L^2(\E)^s $, and $\lsj, \musj \in \R^{s \Nin}$ solve the stationary operator equations
\begin{subequations}
\label{eqn_disc_p0_reg_RK}
\begin{alignat}{6}
  \fA^{-1}\Dt \psj &\,+\, & a\psj &\,-\, & \calK^* \msj &\,+\, & \Ce^*\lsj &\,+\, & \Ce^\ast \musj &= \gs_j-\Ci^*\rs_j  &&\qquad\text{in }[H^1(\bE)^s]^*, \label{eqn_disc_p0_reg_RK_a}\\
  & & \calK \psj &\,+\, & d \msj & & & & &= \fs_j  &&\qquad\text{in }[L^2(\E)^s ]^*, \label{eqn_disc_p0_reg_RK_b}\\
  & & \Ce \psj & & & & &\,-\, & M \musj &= \hs_j &&\qquad\text{in }\R^{s \Nin}, \label{eqn_disc_p0_reg_RK_c}\\
  \Ce \fA^{-1}\Dt \psj & & & & & & & & &= \dot{\hs}_j &&\qquad\text{in }\R^{s \Nin}. \label{eqn_disc_p0_reg_RK_d}
\end{alignat}
\end{subequations}
Therein, the discrete derivative is given by~$\Dt \psj := (\psj - \pjm \ones)/\tau$ and the operators act componentwise. As in Lemma~\ref{lem_disc_qualti} one proves the existence of a solution of the iterative scheme~\eqref{eqn_disc_p0_reg_RK} by inserting~\eqref{eqn_disc_p0_reg_RK_b} into~\eqref{eqn_disc_p0_reg_RK_a} and using~\cite[Lem.~5.6]{AltZ18}, where the initial value is given by~$p_{0,0}=p_0(0)=p(0)$.

For the convergence analysis we assume that the right-hand sides of~\eqref{eqn_op_B_noEps_index1} are continuous and that the components of the right-hand sides of~\eqref{eqn_disc_p0_reg_RK} are given via function evaluations at time points $t_{j-1}+\tau c_i$, $i=1, \dots, s$. For~$\gs_j$ this means exemplary
\[ 
  \gs_j = \Big[ \g(t_{j-1} + \tau \fc_1),\ \dots,\  
  \g(t_{j-1} + \tau \fc_s) \Big]^T.
\]
Furthermore, we suppose that the Runge-Kutta method is of order~$\pOrd$ with stage order~$\qOrd$ which means 
\begin{equation}
\label{eqn_RK_orders}
  \sum_{j=1}^s \fb_j \fc_j^{k-1} = \frac{1}{k}, \quad k=1,\ldots,\pOrd \quad \text{ and }\quad
  \sum_{j=1}^s \fA_{ij} \fc_j^{k-1} = \frac{\fc_i^k}{k}, \quad k=1,\ldots,\qOrd
\end{equation}
with $i=1,\ldots,s$ and~$\pOrd\geq \qOrd+1$. We decompose~$\pj$ as in the proof of Lemma~\ref{lem_disc_parabolic_estimate} into~$\pjt\in H^1_0(\bE)$ and its remainder in~$H^1_{\text{c}}$. Note that we have~$\pjt = \fb^T \fA^{-1} \psjt$ with
\[ 
  \fA^{-1}\Dt \psjt +(\wlap + a) \psjt 
  = \gs_j-\Ci^*\rs_j + \calK^\ast d^{-1} \fs_j - \Ce^- \dot{\hs}_j\qquad 
  \text{in }[H^1_0(\bE)^s]^*. 
\]
A corresponding convergence analysis for this system can be found in~\cite[Ch.~1]{LubO95}. For the part in~$H^1_{\text{c}}$, we obtain by~$\Ce p_{0,0} = \h(0)$, a successive application of equation~\eqref{eqn_disc_p0_reg_RK_d}, and several Taylor expansions the formal expression
\begin{align*}
  \Ce \pj
  &\overset{\hphantom{\eqref{eqn_RK_orders}}}{=} \fb^T \fA^{-1} \Ce \psj
  = \Ce p_{0,j-1} + \tau \fb^T \dot{\hs}_j\\
  &\overset{\hphantom{\eqref{eqn_RK_orders}}}{=} \Ce p_{0}(0) + \tau \sum_{k=1}^j \sum_{i=1}^s \fb_i \dot{\h}(t_{k-1} + \tau \fc_i) \\
  &\overset{\hphantom{\eqref{eqn_RK_orders}}}{=} \h(0) + \tau \sum_{k=1}^j \sum_{i=1}^s \fb_i \sum_{\ell =0}^{\pOrd-1} \frac{(\tau \fc_i)^\ell}{\ell!} \h^{(\ell + 1)}(t_{k-1}) +\sum_{k=1}^j R_k\\
 &\overset{\eqref{eqn_RK_orders}}{=} - \sum_{k=1}^{j-1} \h(t_{k}) + \sum_{k=1}^j \sum_{\ell =0}^{\pOrd} \frac{\tau^\ell}{\ell!} \h^{(\ell)}(t_{k-1}) + \sum_{k=1}^j R_k
  = h(t_j) + \sum_{k=1}^j \big( R_k - \widetilde{R}_k \big)
\end{align*}
with the correction terms $R_k$ and $\widetilde{R}_k$, $k=1,\ldots,j$, given by
\[ 
  R_{k} 
  = \tau \sum_{i=1}^s \fb_i \int_{t_{k-1}}^{t_{k-1} + \tau \fc_i} \frac{(t_{k-1}+\tau \fc_i -s)^{\pOrd-1}}{(\pOrd-1)!}h^{(\pOrd+1)}(s) \ds, \  \widetilde{R}_k 
  = \int_{t_{k-1}}^{t_{k}} \frac{(t_k - s)^{\pOrd}}{\pOrd!}h^{(\pOrd+1)}(s) \ds.
\]
Note that $0\leq \fc_i \leq 1$ for all $i=1,\ldots,s$ implies the estimate 
\begin{equation}
\label{eqn_est_correction_terms_RK}
  \Big| \sum_{k=1}^j \big(R_k - \widetilde{R}_k\big) \Big|^2 
  \leq j\, \sum_{k=1}^j \big| R_k - \widetilde{R}_k \big|^2
  \lesssim t_j\, \tau^{2 \pOrd} \int_0^{t_j} | h^{(\pOrd+1)}(s)|^2 \ds. 
\end{equation}
We emphasize that the included constant only depends on~$\fb$, $\fc$, and~$\pOrd$. Finally, assuming sufficient regularity of the data and the solution, we obtain with~\cite[Th.~1.1]{LubO95} and~\eqref{eqn_est_correction_terms_RK} the error estimate 
\begin{equation}
\label{eqn_est_RK}
\begin{split}
&\| \pj - p_0(t_j)\|^2 + \sum_{k=1}^j \tau\, \| p_{0,k} - p_0(t_k)\|_{H^1(\bE)}^2\\
  & \qquad\qquad \lesssim \big(\tau^{\qOrd+1} \big)^2 \int_{0}^{t_j} \|\pot^{\, (\qOrd+1)}\|_{H^{1}(\bE)}^2 + \|\pot^{\, (\qOrd+2)}\|_{H^{-1}(\bE)}^2 + | h^{(\qOrd+2)}|^2 \ds.
\end{split}
\end{equation}
Pointwise estimates for~$m$ and $\mu$ follow by~\eqref{eqn_disc_p0_reg_RK_b}, \eqref{eqn_disc_p0_reg_RK_c}, and~\eqref{eqn_est_correction_terms_RK}. 

To construct time-continuous approximations of $p$, $m$, and $\mu$, we use~$\pj$, $\mj$, and~$\muj$ as interpolation data. Let~$\potPol{0}{\nOrd}{,\tau}$ denote the continuous function, whose restriction to the interval~$(t_{j-1},t_{j})$ equals the interpolation polynomial of degree~$\nOrd$ in the~$\nOrd+1$ points~$(t_{k},p_{0,k})$ for~$k= j-(\nOrd+1)/2, \dots, j+(\nOrd-1)/2$ if~$\nOrd$ is odd and $k= j-1-\nOrd/2, \dots, j-1+\nOrd/2$ otherwise. Hence, we consider~$\nOrd+1$ interpolation points preferably near the interval~$(t_{j-1},t_{j})$, cf.~the illustration in Figure~\ref{fig_interpolation}. If~$j$ is either too small or too large such that the construction is not possible, then we consider the interpolation in the first (respectively last)~$\nOrd+1$ time points. 
%
\begin{figure}
\centering
\begin{tikzpicture}[scale=0.90]
	\draw[thick] (0.2, 0) -- (3.5, 0);
	\draw[thick] (4.5, 0) -- (8.5, 0);
	\draw[thick, ->] (9.5, 0) -- (13.3, 0);
	\node at (13.7, 0) {$t$};
	\foreach \x in {1,2,3,5,6,7,8,10,11,12} {
		\draw[thick] (\x, 0.1) -- (\x, -0.1);
	}
	\node at (2, -0.37) {$t_{j-1-\nOrd/2}$};
	\node at (6, -0.37) {$t_{j-1}$};
	\node at (7, -0.37) {$t_j$};	
	\node at (8, -0.37) {$t_{j+1}$};
	\node at (11, -0.37) {$t_{j-1+\nOrd/2}$};	
	\draw[opacity=.3,domain=0.3:13,smooth,variable=\x,myBlue2,very thick, dashed] plot ({\x},{sin(\x*62)-cos(\x*32)-0.4+exp(-0.4*\x)+exp(-abs(\x/2-5))}); 
	\draw[domain=6:7,smooth,variable=\x,myBlue2,very thick] plot ({\x},{sin(\x*62)-cos(\x*32)-0.4+exp(-0.4*\x)+exp(-abs(\x/2-5))});	   
    \fill[opacity=.3, myBlue2] (2, 0.45) circle(2.5pt);
    \draw[myBlue2] (2, 0.45) circle(2.5pt);
    \fill[opacity=.3, myBlue2] (3, -0.08) circle(2.5pt);
	\draw[myBlue2] (3, -0.08) circle(2.5pt);
    \fill[opacity=.3, myBlue2] (5, 0) circle(2.5pt);
	\draw[myBlue2] (5, 0) circle(2.5pt);
    \fill[opacity=.3, myBlue2] (6, 1.0) circle(2.5pt);
	\draw[myBlue2] (6, 1.0) circle(2.5pt);
    \fill[opacity=.3, myBlue2] (7, 1.57) circle(2.5pt);
	\draw[myBlue2] (7, 1.57) circle(2.5pt);
    \fill[opacity=.3, myBlue2] (8, 0.94) circle(2.5pt);
	\draw[myBlue2] (8, 0.94) circle(2.5pt);
	\fill[opacity=.3, myBlue2] (10, -1.2) circle(2.5pt);
	\draw[myBlue2] (10, -1.2) circle(2.5pt);
	\fill[opacity=.3, myBlue2] (11, -1.38) circle(2.5pt);
	\draw[myBlue2] (11, -1.38) circle(2.5pt);
\end{tikzpicture} 
\caption{Illustration of the function~$\potPol{0}{\nOrd}{,\tau}$ and its construction on the interval~$(t_{j-1},t_{j})$ for even~$\nOrd$. This part is defined via the interpolation polynomial of degree~$\nOrd$ in the~$\nOrd+1$ time points~$t_{j-1-\nOrd/2}, \dots, t_{j-1+\nOrd/2}$.}
\label{fig_interpolation}
\end{figure}
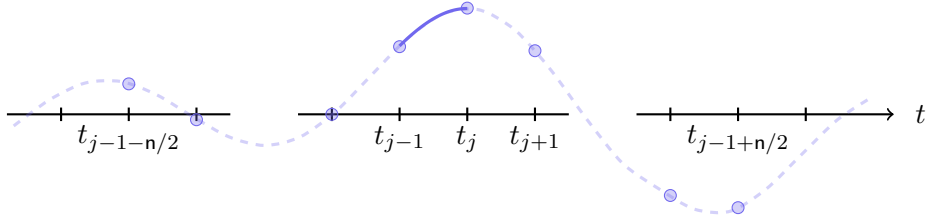
\begin{theorem}\label{thm_RK_p0m0}
Suppose a uniform partition in time with step size~$\tau$ and $m(0) = m_0(0)$. Assume that the data and~$p_0$ are sufficiently smooth. We consider a Runge-Kutta scheme of order~$\pOrd$ and stage order~$\qOrd$, $\pOrd \geq \qOrd +1$, which is algebraically stable, has an invertible coefficient matrix~$\fA$, and satisfies~$\fb^T \fA^{-1} \ones =1$. Then, the piecewise polynomial functions~$\potPol{0}{\qOrd}{,\tau}$ and~$\motPol{0}{\qOrd}{,\tau}$, defined through interpolation data taken from~\eqref{eqn_disc_p0_reg_RK_step}, 
approximate the solution to~\eqref{eqn_op_B_eps} in the sense that 
\begin{align*}
  \Vert p - \potPol{0}{\qOrd}{,\tau} \Vert_{L^2(0,T;H^1(\bE)) \cap C(0,T;L^2(\E))}
  + \Vert m - \motPol{0}{\qOrd}{,\tau} \Vert_{L^2(0,T;L^2(\E))}
  \le (\eps+\tau^{\qOrd+1})\, C_\text{data}.
\end{align*}
Therein, $C_\text{data}$ denotes once more a constant only depending on the data. 
\end{theorem}
\begin{proof}
We only prove the estimate $\Vert p_0 - \potPol{0}{\qOrd}{,\tau} \Vert_{L^2(0,T;H^1(\bE)) \cap C(0,T;L^2(\E))} \leq \tau^{\qOrd+1} C_\text{data}$. The remainder is then similar to the proof of Theorem~\ref{thm_euler_p0m0}. 

Let $L_i$ denote the $i$th Lagrange polynomial on the interval $[0,1]$ corresponding to~$\qOrd +1$ equidistant interpolation nodes, $i=0,\ldots,\qOrd$. By a rescaling of these polynomials we can express~$\potPol{0}{\qOrd}{,\tau}$ in the form 
\[
  \potPol{0}{\qOrd}{,\tau}(t)|_{(t_{j-1},t_j)} 
  = \sum\limits_{i=0}^{\qOrd} L_i\Big(\tfrac{t-t_{k_j}}{\qOrd\, \tau} \Big) p_{0,k_j+i}
\]
for $j=1,\ldots,n$. Therein, $k_j$ denotes the integer such that~$t_{k_j}$ equals the first time point used within the construction of~$\potPol{0}{\qOrd}{,\tau}$ on the time interval~$(t_{j-1},t_j)$. Furthermore, let~$\potPol{0}{\qOrd}{}$ be defined as $\potPol{0}{\qOrd}{,\tau}$ but with interpolation data~$p_0(t_j)$, $j=1,\ldots,n$. Then, it holds that 
\begin{align*}
  \| \potPol{0}{\qOrd}{,\tau} - \potPol{0}{\qOrd}{}\|^2_{C(0,T;L^2(\E))}
  &= \adjustlimits\max_{j= 1,\ldots,n} \sup_{t\in (t_{j-1},t_j)} \big\| \sum_{i=0}^{\qOrd} L_i\Big(\tfrac{t-t_{k_j}}{\qOrd\, \tau} \Big) \big(p_{0,k_j+i}-p_0(t_{k_j+i})\big) \big\|^2_{L^2(\E)}\\
  &\leq (\qOrd + 1) \sum_{i=0}^{\qOrd} \| L_i \|^2_{C([0,1])} \max_{j= 1,\ldots,n} \| p_{0,j}-p_0(t_{j})\|^2_{L^2(\E)}.
\end{align*}
With the standard error bound from polynomial interpolation and the embedding of $W(0,T;H^1_0(\bE),H^{-1}(\bE))$ in $C(0,T;L^2(\E))$ we get 
\begin{align*}
  \| \potPol{0}{\qOrd}{}- p_0\|^2_{C(0,T;L^2(\E))} 
  &\leq \frac{(\tau^{\qOrd+1})^2 }{16 (\qOrd+1)^2} \|p_0^{(\qOrd+1)}\|^2_{C(0,T;L^2(\E))}\\
  &\lesssim \big(\tau^{\qOrd+1} \big)^2  \int_{0}^{T} \|p_0^{\, (\qOrd+1)}\|^2_{H^1(\bE)} + \|\pot^{\, (\qOrd+2)}\|_{H^{-1}(\bE)}^2 + |h^{(\qOrd+2)}|^2 \ds.
\end{align*}
For the estimate in $L^2(0,T;H^1(\bE))$ we note that 
\begin{align*}
  \int_0^T \| \potPol{0}{\qOrd}{,\tau} - \potPol{0}{\qOrd}{}\|^2_{H^1(\bE)} \ds 
  &\leq (\qOrd +1) \sum_{j=1}^n \sum_{i=0}^{\qOrd} \int_{t_{j-1}}^{t_j} L_i^2\Big(\tfrac{s-t_{k_j}}{\qOrd\, \tau} \Big) \ds\, \| p_{0,k_j+i}-p_0(t_{k_j+i})\|^2_{H^1(\bE)}\\
  &\leq (\qOrd +1)^3 \sum_{i=1}^{\qOrd} \int_0^1 L_i^2(s) \ds\, \sum_{j=1}^n \tau \| p_{0,j}-p_0(t_j)\|^2_{H^1(\bE)}
\end{align*}
and with a well-known finite element estimate, cf.~\cite[Ch.~II, Rem.~6.5]{Bra07}, we have
\[ 
  \int_0^T \| \potPol{0}{\qOrd}{} - p_0\|^2_{H^1(\bE)} \ds 
  \lesssim (\tau^{\qOrd+1})^2 \int_0^T \| p_0^{(\qOrd +1)} \|^2_{H^1(\bE)} \ds.
\]
The claimed bound for~$\potPol{0}{\qOrd}{,\tau} - p_0$ finally follows by the four inequalities above, the triangular inequality, and estimate~\eqref{eqn_est_RK}.
\end{proof}
%
%
\subsection{Second-order approximation}\label{sect:timeInt:second}
In this subsection we make use of the second-order approximation from Section~\ref{sect:expansion:second}, i.e., we derive a time discretization for the approximation of~$p$ and~$m$ based on~$\hat p$ and $\hat m$, respectively. For this, we consider the coupled system given by~\eqref{eqn_op_B_noEps} and~\eqref{eqn_op_B_approx_1}. Recall that the latter system includes~$\dot m_0$ as a right-hand side. This may cause computational issues, since~$m_0$ is an 'algebraic' variable in~\eqref{eqn_op_B_noEps}. For the corresponding finite-dimensional case it is well-known that this reduces the convergence order, cf.~\cite[Ch.~6~f.]{HaiW96}. 

To prevent such a loss of convergence, we replace~$\dot m_0$ in~\eqref{eqn_op_B_approx_1_b} by the derivative of~\eqref{eqn_op_B_noEps_b}, which is possible under the assumptions of Lemma~\ref{lem_weak_op_B_noeps_reg}. In place of~\eqref{eqn_op_B_approx_1_b} this leads to an equation independent of~$\dot m_0$. Similar to the discussion in Remark~\ref{rem:Dh_doth}, this achieves the inclusion of~$\dot \f$ rather then the discrete derivative of~$\f$. The latter would lead to a perturbed right-hand side and thus, to declined convergence properties~\cite{Alt15}.

The resulting coupled system includes two constraints. From~\eqref{eqn_op_B_noEps_c} we have~$\Ce p_0 = \h$, for which we apply the same regularization step as in the previous subseciton. Second, we have the condition~$\Ce p_1 = 0$  from~\eqref{eqn_op_B_approx_1_c}, which we incorporate into the solution space of~$p_1$. In summary, this leads to 
\begin{subequations}
	\label{eqn_op_B_coupled}
	\begin{align}
	\dot p_0 + a p_0 - \calK^* m_0  + \Ce^*\lambda_0 + \Ce^\ast \mu_0 &= \g-\Ci^*\r  &&\text{in }[H^1(\bE)]^*, \label{eqn_op_B_coupled_a}\\
	\dot p_1 + a p_1 - \calK^* m_1  &= 0  &&\text{in } H^{-1}(\bE), \label{eqn_op_B_coupled_b}\\
	- d^{-1} \calK \dot p_0 + \calK p_1 + d m_1 &= - d^{-1} \dot{\f}  &&\text{in }[L^2(\E)]^*, \label{eqn_op_B_coupled_c}\\	
	\calK p_0 + d m_0 &= \f  &&\text{in }[L^2(\E)]^*, \label{eqn_op_B_coupled_d}\\
	\Ce p_0  - M \mu_0 &= \h &&\text{in }R^{\Nin}, \label{eqn_op_B_coupled_e}\\
	\Ce \dot p_0 &= \dot{\h} &&\text{in }R^{\Nin} \label{eqn_op_B_coupled_f}
	\end{align}
\end{subequations}
with an artificial Lagrange multiplier~$\mu\colon [0,T] \to \R^{\Nin}$ and an arbitrary invertible matrix~$M \in \R^{\Nin \times \Nin}$. 
\subsubsection{Implicit Euler scheme}
We first concentrate on the discretization by the implicit Euler scheme. The discretized version of~\eqref{eqn_op_B_coupled} with step size $\tau$ then reads   
\begin{subequations}
\label{eqn_disc_p1_reg}
\begin{align}
	\Dt \pj + a \pj - \calK^* \mj + \Ce^*\lj + \Ce^\ast \muj &= \g_j-\Ci^*\r_j  &&\text{in }[H^1(\bE)]^*, \label{eqn_disc_p1_reg_a}\\
	\Dt \pj[1] + a \pj[1] - \calK^* \mj[1]  &= 0  &&\text{in } H^{-1}(\bE), \label{eqn_disc_p1_reg_b}\\
	- d^{-1} \calK (\Dt \pj) + \calK \pj[1] + d \mj[1] &= - d^{-1} \dot{\f}_j  &&\text{in }[L^2(\E)]^*, \label{eqn_disc_p1_reg_c}\\	
	\calK \pj + d \mj &= \f_j  &&\text{in }[L^2(\E)]^*, \label{eqn_disc_p1_reg_d}\\
	\Ce \pj  - M \muj &= \h_j &&\text{in }R^{\Nin}, \label{eqn_disc_p1_reg_e}\\
	\Ce \Dt \pj &= \dot{\h}_j &&\text{in }R^{\Nin}. \label{eqn_disc_p1_reg_f}
\end{align}
\end{subequations}
We emphasize that equations~\eqref{eqn_disc_p1_reg_a} and~\eqref{eqn_disc_p1_reg_d}-\eqref{eqn_disc_p1_reg_f} equal the Euler discretization of the regularized first-order system~\eqref{eqn_disc_p0_reg}. Thus, a decoupling of the system would imply the existence of~$\pj$, $\mj$, $\lj$, and $\muj$ in every time step, see~Lemma~\ref{lem_disc_qualti}. The assumption~$p_0(0)\in H^1(\bE)$ then implies~$\Dt \pj \in H^1(\bE)$ for $j=1,\ldots,n$. This, in turn, ensures that the term~$d^{-1} \calK (\Dt \pj)$ is well-defined in $[L^2(\E)]^*$ and thus, the unique solvability of~\eqref{eqn_disc_p1_reg} in each time-step, cf.~\cite[Th.~4.1]{EmmT10}. 

Now, let~$\potC[1]$, $\motC[1]$ denote the piecewise constant functions and~$\potL[1]$ the continuous and piecewise linear function defined via~$\pj[1]$ and~$\mj[1]$, respectively. 
\begin{lemma}\label{lem_disc_parabolic_estimate_coupled}
Suppose right-hand sides $\f \in H^2(0,T;L^2(\E))$, $\g \in H^2(0,T;H^{-1}(\bE))$ with~$\g \in L^2(0,T;[H^1(\bE)]^\ast)$, $\h \in H^3(0,T;\R^{\Nin})$, and~$\r \in H^2(0,T;\R^{\Nout})$. Further assume consistent initial data $p_0(0) \in H^1(\bE)$ with $\Ce p_0(0) = \h(0)$, the existence of a function~$\pstar\in H^1_0(\bE)$ satisfying~\eqref{def_pstar}, and~$\pstar^\prime\in L^2(\E)$ such that 
\[
  (\pstar^\prime, q) 
  = \big\langle \dot \g(0) -\Ci^\ast \dot \r(0) - (\wlap +a) (\pstar + \Ce^- \dot \h (0)) - \ddot \h (0) + \calK^\ast d^{-1}\dot \f(0), q \big\rangle 
\]
for all $q\in H^1_0(\bE)$. If the right-hand sides in~\eqref{eqn_disc_p1_reg} are defined through function evaluations, then there exists a constant~$C_\text{data}$ only depending on the data with  
\[
  \|\potC[1]  - p_1\|^2_{L^2(0,T;H^1(\bE))} 
  + \|\potL[1] - p_1\|^2_{C(0,T;L^2(\E))} 
  + \|\motC[1] - m_1  \|^2_{L^2(0,T;L^2(\E))} 
  \leq \tau^2\, C_\text{data}. 
\]
\end{lemma}
\begin{proof}
By Lemma~\ref{lem_weak_op_B_noeps_reg} we conclude that the solution $(p_0,p_1,m_0,m_1)$ satisfies
\begin{align*}
p_0 &\in H^3(0,T;H^{-1}(\bE)) \cap H^2(0,T;H^1(\bE)) \cap C^2(0,T;L^2(\E)), &  m_0 &\in H^2(0,T;L^2(\E)),\\
p_1 &\in H^2(0,T;H^{-1}(\bE)) \cap H^1(0,T;H^1(\bE)) \cap C^1(0,T;L^2(\E)), &  m_1 &\in H^1(0,T;L^2(\E)).
\end{align*}
Following the notation of the proof of Lemma~\ref{lem_disc_parabolic_estimate} and considering~$\Ce \pj[1] = 0$, we have~$\pj[1] = \pjt[1]$. Further, proceeding as in the derivation of~\eqref{eqn_disc_pointwise}, we obtain by~\eqref{eqn_disc_p1_reg_b} and~\eqref{eqn_disc_p1_reg_c} the estimate 
\begin{align}
  \| \pj[1] - p_1&(t_j)\|^2 + \cLap \sum_{k=1}^j \tau\, \| p_{1,k} - p_1(t_k)\|_{H^1(\bE)}^2 \notag \\
  \leq\ & \frac{2\tau^2}{3\cLap}  \int_{0}^{t_j} \|\ddot{p}_{1}(s)\|^2_{H^{-1}(\bE)} \ds + \frac{2\tau}{\cLap} \sum_{k=1}^j\, \big\| \calK^\ast d^{-2} \calK(\Dt \pkt - \dot \pot(t_k)) \big\|^2_{H^{-1}(\bE)}. \label{eqn_disc_pointwise_p1}
\end{align}
Note that we have the additional term $\Dt \pkt - \dot \pot(t_k)$, since $-d^{-1}\calK \dot p_0 = -d^{-1}\calK (\dot \pot +\Ce^- \dot h)$ in~\eqref{eqn_op_B_coupled_c} can be seen as a part of the right-hand side, which is not evaluated at $t_j$ but approximated by $-d^{-1}\calK \Dt p_{0,j}=-d^{-1}\calK(\Dt \pjt + \Ce^- \dot h(t_j))$. 

To estimate  the second part of the right-hand side of~\eqref{eqn_disc_pointwise_p1}, we realize that the operator $\calK^\ast d^{-2} \calK$ is linear and bounded from $H^1_0(\bE)$ to $H^{-1}(\bE)$. Furthermore, with a short calculation one proves that the discrete derivative $\Dt \pjt$ satisfies
\[ 
  \Dt (\Dt \pjt) + (\wlap + a) \Dt \pjt = \Dt (\g_j - \Ci^\ast \r_j + \calK^\ast d^{-1} f_j - \Ce^- \dot \h_j) 
  =: \Dt \widetilde{\F}_j \quad \text{in } H^{-1}(\bE)
\]
with $\widetilde{\F}_0 = \widetilde{\F}(0)$ and $\Dt \pt_{0,0} = \pstar$. Similar to~\eqref{eqn_disc_pointwise} and~\eqref{eqn_disc_pointwise_p1} we then get the estimate
\begin{align*}
  \| \Dt \pjt - \dot \pot&(t_j)\|^2 + \cLap \sum_{k=1}^j \tau\, \| \Dt \pkt - \dot \pot(t_k)\|_{H^1(\bE)}^2 \\
  \leq\,& \frac{2\tau^2}{3\cLap}  \int_{0}^{t_j} \|\pot^{(3)}(s)\|^2_{H^{-1}(\bE)} \ds + \frac{2\tau}{\cLap} \sum_{k=1}^j \| \Dt \widetilde{\F}_k - \dot{\widetilde{\F}}(t_k) \|^2_{H^{-1}(\bE)}.
\end{align*}
The claim finally follows by the steps of the proof of Lemma~\ref{lem_disc_parabolic_estimate} together with the estimate
\begin{equation*}
\| \Dt \widetilde{\F}_k - \dot{\widetilde{\F}}(t_k) \|^2_{H^{-1}(\bE)} 
  = \Big\| \frac{1}{\tau} \int_{t_{k-1}}^{t_k} (s-t_{k-1}) \ddot{\widetilde{\F}}(s) \ds \Big\|^2_{H^{-1}(\bE)} \leq \frac{\tau}{3} \int_{t_{k-1}}^{t_k} \| \ddot{\widetilde{\F}}(s)\|^2_{H^{-1}(\bE)} \ds. \qedhere
\end{equation*}
\end{proof}
Recall the definition of~$\hat p$ and~$\hat m$ from Section~\ref{sect:expansion}. Similarly, we now define their piecewise constant and piecewise linear approximations given by the Euler scheme~\eqref{eqn_disc_p1_reg}, i.e.,    
\[
  \phtC := \potC + \eps \potC[1], \qquad
  \mhtC := \motC + \eps \motC[1], \qquad
  \phtL := \potL + \eps \potL[1].
\]
This then leads to the following approximation result. 
\begin{theorem}\label{thm_euler_p1m1}
Let the assumptions of Theorem~\ref{thm_p1m1}, Lemma~\ref{lem_disc_parabolic_estimate}, and Lemma~\ref{lem_disc_parabolic_estimate_coupled} be satisfied together with~$m(0) = \hat m(0)$. Then, the piecewise constant functions~$\phtC, \mhtC$ and the piecewise linear function~$\phtL$ defined through~\eqref{eqn_disc_p1_reg} satisfy 
\begin{align*}
  \Vert p - \phtC \Vert_{L^2(0,T;H^1(\bE))}
  + \Vert m - \mhtC \Vert_{L^2(0,T;L^2(\E))}
  + \Vert p - \phtL \Vert_{C(0,T;L^2(\E))}
  \le (\eps^2 + \tau)\, C_\text{data}
\end{align*}
with a constant~$C_\text{data}$ only depending on the data. 	
\end{theorem}
\begin{proof}
Again we simply combine the previous results. For~$\phtL$ we obtain  
\[
  \Vert (p - \phtL)(t) \Vert
  \le \Vert (p - \hat p)(t) \Vert + \Vert (p_0 - \potL)(t) \Vert + \Vert \eps (p_1 - \potL[1])(t) \Vert
  \le (\eps^2 + \tau + \eps \tau)\, C_\text{data}.
\]
The proofs for the two remaining parts are similar. 
\end{proof}
As for the first-order approximation in Section~\ref{sect:timeInt:first}, we finish this section with a discussion of the discretization by algebraically stable Runge-Kutta methods. 
\subsubsection{Runge-Kutta methods}\label{sect:timeInt:second:RK}
For the analysis of the approximation order of a general Runge-Kutta method applied to~\eqref{eqn_op_B_coupled}, we need to estimate the error between~$p_1(t_j)$ and its time-discrete counterpart~$\pj[1]$. As for implicit Euler scheme, we have to keep in mind that in the Runge-Kutta discretization of~\eqref{eqn_op_B_coupled_c} the term $\dot{\ps}_{0,j}= [ \dot p_0(t_{j-1} + \tau \fc_1),\ \dots,\ \dot p_0(t_{j-1} + \tau \fc_s)]^T$ is only approximated by $\fA^{-1} D_\tau \psj$. For the perturbation~$\fA^{-1} D_\tau \psj - \dot{\ps}_{0,j}$ we have with equation~\eqref{eqn_op_B_coupled_f} and Taylors theorem 
\begin{align*}
  \big(\fA^{-1} D_\tau \psj - \dot{\ps}_{0,j}\big)_i 
  &= \big(\fA^{-1} D_\tau \psjt - \dot{\pst}_{0,j}\big)_i\\
  & =  \frac{1}{\tau} \sum_{\ell=1}^s \fA^{-1}_{i\ell}\big[\big((\psjt)_\ell - \pt_0(t_{j-1}+\tau \fc_\ell)\big)- \big(\pjmt-\pt(t_{j-1})\big)\big] -R_{j,i}
\end{align*}
for $i=1,\ldots,s$. With~\eqref{eqn_RK_orders} the correction term~$R_{j,i}$ is given by
\begin{align*}
R_{j,i} = &\int_{t_{j-1}}^{t_{j-1} + \tau \fc_i} \frac{(t_{j-1} + \tau \fc_i -s)^{\qOrd-1}}{(\qOrd-1)!}\pt_0^{\,(\qOrd+1)}(s) \ds\\
 & \qquad \qquad - \sum_{\ell=1}^s \frac{\fA^{-1}_{i \ell}}{\tau}  \int_{t_{j-1}}^{t_{j-1} + \tau \fc_\ell} \frac{(t_{j-1} + \tau \fc_\ell -s)^{\qOrd}}{\qOrd!}\pt_0^{\,(\qOrd+1)}(s) \ds.
\end{align*}
Next, we combine~\cite[p.~605, Rem.~{(c)}]{LubO95}, estimate~\eqref{eqn_est_RK}, and 
\[
  \sum_{k=1}^j \tau\, \sum_{i=1}^s \| (\pskt)_i - \pt_0(t_{k-1} + \fc_i\tau)\|_{H^1(\bE)}^2 
  \lesssim \big(\tau^{\qOrd+1} \big)^2 \int_{0}^{t_j} \|\pot^{\, (\qOrd+1)}\|_{H^{1}(\bE)}^2 + \|\pot^{\, (\qOrd+2)}\|_{H^{-1}(\bE)}^2 \ds,
\]
which follows from \cite[Th.~1.1]{LubO95}. Together with~$p_1\in H^1_0(\bE)$ this then yields
\begin{align*}
&\| \pj[1] - p_1(t_j)\|^2 + \sum_{k=1}^j \tau\, \| p_{1,k} - p_1(t_k)\|_{H^1(\bE)}^2\\
  \lesssim\, & \tau^{2\qOrd}  \int_{0}^{t_j} \|p_1^{\, (\qOrd)}\|_{H^{1}(\bE)}^2 + \|p_1^{\, (\qOrd+1)}\|_{H^{-1}(\bE)}^2\ds + \sum_{k=1}^j \tau \| \calK^\ast d^{-2} \calK(\fA^{-1} D_\tau \psk - \dot{\ps}_{0,k})\|_{H^{-1}(\bE)}^2\\
  \leq\, &  \tau^{2\qOrd} \int_{0}^{t_j} \|\pot^{\, (\qOrd+1)}\|_{H^{1}(\bE)}^2 + \|p_1^{\, (\qOrd)}\|_{H^{1}(\bE)}^2 + \|\pot^{\, (\qOrd+2)}\|_{H^{-1}(\bE)}^2 + \|p_1^{\, (\qOrd+1)}\|_{H^{-1}(\bE)}^2\ds    + | h^{(\qOrd+2)}|^2 \ds.
\end{align*}
For the approximation of $\hat p$ and $\hat m$ we use once more piecewise polynomials. Since the approximation of $p_1$ by $\pj[1]$ is one order smaller then the approximation of $p_0$ by $\pj$, we consider 
\[ 
  \phtPol{\qOrd} := \potPol{0}{\qOrd}{,\tau} + \eps \potPol{1}{\qOrd-1}{,\tau}, \qquad \mhtPol{\qOrd} := \potPol{0}{\qOrd}{,\tau} + \eps \motPol{1}{\qOrd-1}{,\tau}. 
\]
Here, $\potPol{1}{\qOrd-1}{,\tau}$ is defined similarly to~$\potPol{0}{\qOrd}{,\tau}$ but with interpolation points~$(t_j,\pj[1])$. Analogously, the piecewise polynomial~$\motPol{1}{\qOrd-1}{,\tau}$ is defined based on~$(t_j,\mj[1])$.
\begin{theorem}
Let the assumptions of Theorem~\ref{thm_euler_p0m0} be satisfied together with~$m(0) = \hat m(0)$. Let $\phtPol{\qOrd}$ and $\mhtPol{\qOrd}$ defined through the Runge-Kutta discretization of~\eqref{eqn_op_B_coupled}. Then it holds that 
\begin{align*}
  \Vert p - \phtPol{\qOrd} \Vert_{L^2(0,T;H^1(\bE)) \cap C(0,T;L^2(\E))}
  + \Vert m - \mhtPol{\qOrd} \Vert_{L^2(0,T;L^2(\E))}
  \le (\eps^2+ \eps \tau^{\qOrd} + \tau^{\qOrd+1} )\, C_\text{data}
\end{align*}
with a constant~$C_\text{data}$ only depending on the data.
\end{theorem}
\begin{proof}
The statement can be shown similarly as Theorem~\ref{thm_euler_p1m1}, in combination with Theorem~\ref{thm_RK_p0m0}.
\end{proof}

%% file: numerics.tex
\section{Numerical Experiments}\label{sect:numerics}
In this final section, we aim to illustrate the proven convergence results as well as the differences in the approximation orders in terms of~$\eps$ in the finite and infinite dimensional case. 
%
\subsection{Convergence in $\tau$}\label{sect:numerics:tau}
In this first experiment we show the performance of the proposed numerical scheme based on the combination of an $\eps$-expansion and well-known time stepping methods. For this, we consider the network shown in Figure~\ref{fig_network} where all pipes have unit length. We assume~$\g\equiv 0$ and~$\f\equiv 1$ on the edges~$e_2$, $e_4$ and zero otherwise. The initial data is given by
\[
p^{e_1}(x)|_{t=0} = x, \qquad
p^{e_j}(x)|_{t=0} = 1,\qquad
p^{e_6}(x)|_{t=0} = 1-x 
\]
for $j=2,\dots, 5$ and, in the case $\eps>0$, 
\[
m^{e_1}|_{t=0} \equiv -1,\qquad
m^{e_2}|_{t=0} = m^{e_4}|_{t=0} = m^{e_6}|_{t=0} \equiv 1,\qquad
m^{e_3}|_{t=0} = m^{e_5}|_{t=0} \equiv 0.
\]
We consider~$\eps=10^{-3}$ and set the damping parameter to~$d\equiv 1$. Note that this implies consistency of the initial condition. Homogeneous Dirichlet boundary conditions for~$p$ are stated in the nodes~$v_1$ and~$v_6$, whereas the coupling conditions are defined through 
\[
  \r(v_2) = -2, \qquad
  \r(v_3) = 0, \qquad  
  \r(v_4) = -1, \qquad  
  \r(v_5) = 1.
\]
Figure~\ref{fig_tau} compares the approximations of first and second order in~$\eps$, 
based on the implicit Euler scheme and the Radau IIa method of third order. As predicted in Section~\ref{sect:timeInt}, the difference of the exact pressure~$p$ and the computed approximation of~$p_0$ measured in the $C(0,1;L^2(\E))$-norm converges linearly (or quadratically in the Radau IIa case, cf.~Section~\ref{sect:timeInt:second:RK}) in time up to the point where the approximation error of the~$\eps$-expansion dominates. Since~$\hat p$ approximates~$p$ with order two in terms of~$\eps$, the overall error achieves a lower level if we include the numerical approximation of~$\eps p_1$. 
For the computations we have used a finite element discretization in space with~$10$ grid points on each pipe. The pressure~$p$ is approximated by~$P_1$ finite elements, whereas the approximation of~$m$ is piecewise constant. 
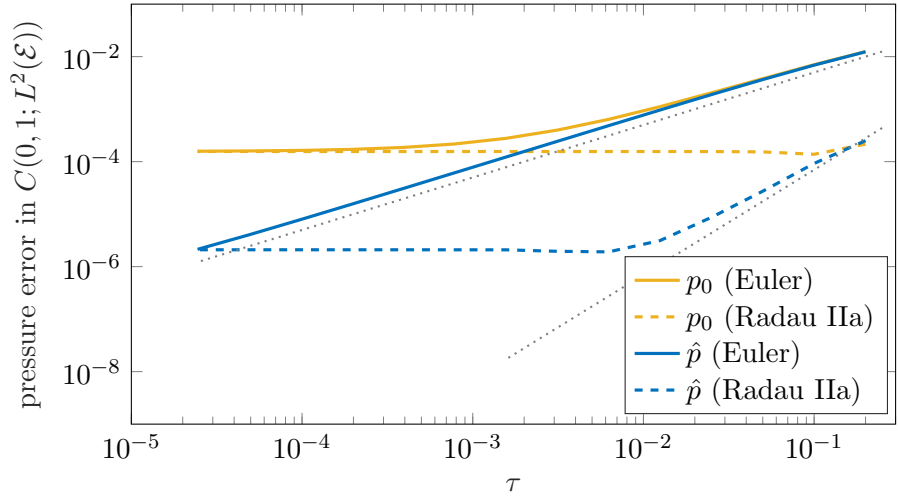
\begin{figure}
\label{fig_tau}	
\input{exp_tau}
\caption{Convergence history for the example of Section~\ref{sect:numerics:tau}. The expected convergence rates can be observed up to the point where the error of $\eps$-expansion dominates. The dotted lines show orders~$1$ and~$2$.}
\end{figure}
%
\subsection{Convergence in $\eps$}
The second numerical experiment illustrates the transition of the approximation order in terms of~$\eps$ from the finite to the infinite dimensional case, cf.~Remark~\ref{rem_finiteinfinite}. Recall that Theorem~\ref{thm_p0m0} implies 
\[
\|p-p_0\|_{C(0,T;\,L^2)} 
\le C_\text{cons} \sqrt{\eps} + \O(\eps)
\]
in the infinite dimensional setting with $C_\text{cons} = 0$ if the initial data is consistent. In finite dimensions one shows~$\|p-p_0\|_{C(0,T;\,\R^n)} = \O(\eps)$, independent of the initial data, see~\cite[Ch.~2.5, Th.~5.1]{KokKO99}. We emphasize that this is an asymptotic result whereas the bound in Theorem~\ref{thm_p0m0} is valid for all~$\eps>0$. 

Considering spatial discretizations of an infinite dimensional system, one expects a smooth transition of the approximation order. To see this, we consider the problem described in~Appendix~\ref{app_sharpness} of a single pipe of unit length with constant damping $d=1$. We assume homogeneous Dirichlet boundary conditions for the pressure and inconsistent initial data given by
\[ 
  p(x,0) = 0 \quad \text{and} \quad 
  m(x,0)= \sum_{k=1}^\infty \frac{\cos(\pi k x)}{\pi k^{0.55}}.
\]
As before, we consider $P_1$ finite elements for the pressure and piecewise constants for the mass flux. To estimate the error $\|p-p_0\|_{C(0,T;\,\R^n)}$ we take the norm which is induced by the mass matrix, i.e., the norm associated to the $L^2(0,1)$-norm. To approximate the order in~$\eps$ depending on the spatial discretization parameter~$h$, we calculate for a fixed step size~$h$ the norms 
\[
  \operatorname{err}(\eps) : =\|p_h(\, \cdot\, ; \eps)-p_{0,h}\|_{C(0,T;\,\R^n)}, \qquad
  \eps= 1/(8\sqrt{2^i}) 
\]
for $i=1,\ldots,16$. Using the ansatz $\operatorname{err}(\eps)= C(h)\,\eps^\alpha$, or equivalently~$\log \operatorname{err}(\eps)= \alpha \log \eps + \log C(h)$, we estimate the exponent~$\alpha$ by a linear least squares problem. The results are illustrated in Figure~\ref{fig_eps}. As expected, the exponential order~$\alpha$ decreases as the spatial mesh size~$h$ becomes smaller. One can clearly see the beginning of an asymptotic behavior with a limit much smaller than~$1$. Following the analysis in Appendix~\ref{app_sharpness}, we expect the limit of the fitted exponential order~$\alpha$ to be~$0.55$.
\begin{figure}
	\label{fig_eps}		
	\input{exp_eps_app}
	\caption{Estimate of the exponential order~$\alpha$ in~$\operatorname{err}(\eps)= C(h)\,\eps^\alpha$ as a function of the spatial mesh size~$h$.}
\end{figure}

%% file: exp_tau.tex
%
%
\definecolor{mycolor1}{rgb}{0,0.447,0.741}%
\definecolor{mycolor3}{rgb}{0.92900,0.69400,0.12500}%
\begin{tikzpicture}

\begin{axis}[%
width=4.0in,
height=2.2in,
scale only axis,
xmode=log,
xmin=1e-05,
xmax=0.3,
xminorticks=true,
xlabel style={font=\color{white!15!black}},
xlabel={$\tau$},
ymode=log,
ymin=1e-9,
ymax=1e-01,
yminorticks=true,
ylabel style={font=\color{white!15!black}},
ylabel={pressure error in $C(0,1;L^2(\E))$},
axis background/.style={fill=white},
title style={font=\bfseries},
legend style={legend cell align=left, align=left, draw=white!15!black, at={(0.99,0.4)}}
]
\addplot[color=mycolor3, very thick]
  table[row sep=crcr]{%
0.2	0.0124936496581605\\
0.1	0.00702667100993455\\
0.05	0.00380745429361191\\
0.025	0.00204287079597019\\
0.0125	0.00111435391685351\\
0.00625	0.000639055708623285\\
0.003125	0.000398522128928422\\
0.0015625	0.000277524966771344\\
0.00078125	0.000216840315550603\\
0.000390625	0.000186451812303794\\
0.0001953125	0.000171246009680683\\
9.765625e-05	0.000163640212960962\\
4.8828125e-05	0.000159836590899132\\
2.44140625e-05	0.000157934598244863\\
};
\addlegendentry{$p_0$ (Euler)}

\addplot[color=mycolor3, very thick, dashed]
  table[row sep=crcr]{%
0.2	0.000215808716391835\\
0.1	0.000138142376310191\\
0.05	0.000153400266856253\\
0.025	0.000155724706324736\\
0.0125	0.000155979657777239\\
0.00625	0.000156025372083365\\
0.003125	0.000156030416623089\\
0.0015625	0.000156032381869654\\
0.00078125	0.000156032445197912\\
0.000390625	0.000156032454893568\\
0.0001953125	0.000156032485628268\\
9.765625e-05	0.000156032485751988\\
4.8828125e-05	0.000156032485767387\\
2.44140625e-05	0.000156032485769495\\
};
\addlegendentry{$p_0$ (Radau IIa)}

\addplot [color=mycolor1, very thick]
  table[row sep=crcr]{%
0.2	0.0123649025725807\\
0.1	0.00688921513620136\\
0.05	0.00366085699669655\\
0.025	0.00189190089444856\\
0.0125	0.000961024159488122\\
0.00625	0.000484506333908571\\
0.003125	0.000243337090630523\\
0.0015625	0.00012201931870917\\
0.00078125	6.11748419376243e-05\\
0.000390625	3.07069365466615e-05\\
0.0001953125	1.5462758052717e-05\\
9.765625e-05	7.84035084668304e-06\\
4.8828125e-05	4.03286438126604e-06\\
2.44140625e-05	2.13573782634654e-06\\
};
\addlegendentry{$\hat p$ (Euler)}

\addplot [color=mycolor1, very thick, dashed]
  table[row sep=crcr]{%
0.2	0.000250709059872132\\
0.1	9.25083828710109e-05\\
0.05	2.73692551904481e-05\\
0.025	8.68765525316826e-06\\
0.0125	3.13419952106706e-06\\
0.00625	1.91129877636189e-06\\
0.003125	1.96006333160555e-06\\
0.0015625	2.09026659416377e-06\\
0.00078125	2.0933359627572e-06\\
0.000390625	2.09522028395319e-06\\
0.0001953125	2.0954837721248e-06\\
9.765625e-05	2.09551885055916e-06\\
4.8828125e-05	2.09552338684882e-06\\
2.44140625e-05	2.09552396349777e-06\\
};
\addlegendentry{$\hat p$ (Radau IIa)}

\addplot[color=gray, thick, dotted]
  table[row sep=crcr]{%
0.25	0.0125\\
0.111111111111111	0.00555555555555556\\
0.0526315789473684	0.00263157894736842\\
0.0256410256410256	0.00128205128205128\\
0.0126582278481013	0.000632911392405063\\
0.00628930817610063	0.000314465408805031\\
0.00313479623824451	0.000156739811912226\\
0.00156494522691706	7.82472613458529e-05\\
0.000781860828772478	3.90930414386239e-05\\
0.000390777647518562	1.95388823759281e-05\\
0.000195350654424692	9.76753272123462e-06\\
9.76657876745776e-05	4.88328938372888e-06\\
4.8830509302212e-05	2.4415254651106e-06\\
2.4414658561e-05	1.22073292805e-06\\
};

\addplot[color=gray, thick, dotted]
  table[row sep=crcr]{%
0.25	0.0004375\\
0.111111111111111	8.64197530864197e-05\\
0.0526315789473684	1.93905817174515e-05\\
0.0256410256410256	4.60223537146614e-06\\
0.0126582278481013	1.12161512578112e-06\\
0.00628930817610063	2.76887781337764e-07\\
0.00313479623824451	6.87886321871837e-08\\
0.00156494522691706	1.71433749427534e-08\\
};

\end{axis}
\end{tikzpicture}%

%% file: exp_eps_app.tex
%
%
%
%
\definecolor{mycolor1}{rgb}{0,0.447,0.741}%
\begin{tikzpicture}

\begin{axis}[%
width=4.0in,
height=2.0in,
scale only axis,
separate axis lines,
every outer x axis line/.append style={darkgray!60!black},
every x tick label/.append style={font=\color{darkgray!60!black}},
xmode=log,
xmin=0.0005,
xmax=0.166666666666667,
xminorticks=true,
xlabel={spatial mesh size $h$},
every outer y axis line/.append style={darkgray!60!black},
every y tick label/.append style={font=\color{darkgray!60!black}},
ymin=0.55,
ymax=0.8,
ylabel={$\text{fitted exponent }\alpha$},
legend style={draw=darkgray!60!black,fill=white,legend cell align=left}
]
\addplot [
color=mycolor1,
very thick,
solid,
forget plot
]
table[row sep=crcr]{
0.166666666666667 0.774071912560396\\
0.0909090909090909 0.714349547832569\\
0.0476190476190476 0.664539118688238\\
0.024390243902439 0.629840041467353\\
0.0123456790123457 0.606529674139336\\
0.0062111801242236 0.590460840349867\\
0.00311526479750779 0.579161728995493\\
0.0015600624024961 0.570638150591638\\
0.0007806401249024 0.563924960309423\\
};
\end{axis}
\end{tikzpicture}%

%% file: appendix_generalization.tex
\section{Generalized System Equations}
\label{app_generalization}
The results in Sections~\ref{sect:formulation} - \ref{sect:timeInt} remain valid for certain generalizations of the considered model~\eqref{eqn_op_B_eps}, which we discuss here. Let $\P$, $\calM$, and $\Q$ denote the Hilbert spaces in which we search for the solution components $p$, $m$, and $\lambda$, respectively. We assume that $\P$ forms a Gelfand triple with pivot space~$\calH$, i.e., $\P\, {\hookrightarrow}\, \calH \cong \calH^\ast\, {\hookrightarrow}\, \P^\ast$ where all embeddings are dense. Then, system~\eqref{eqn_op_B_eps} may be generalized to
\begin{equation*}
\begin{alignedat}{5}
	\dot p&\ +\ &\A p &\ -\ &\calK^\ast m & + \Ce^*\lambda\ &=&\ \g_1 + \g_2 &&\qquad \text{in }\P^*,\\
	\eps\, \dot m&\ +\ &\calK p &\ +\ &\D m & &=&\ \f  &&\qquad \text{in }\calM^*,\\
	& & \Ce p & & & &=&\ \h &&\qquad \text{in } \Q^\ast
\end{alignedat}
\end{equation*}
with linear and bounded operators
\[
  \A\colon \calH \to \calH^\ast, \qquad 
  \calK\colon \P \to \calM^\ast, \qquad
  \D\colon \calM \to \calM^\ast, \qquad 
  \Ce \colon \P\to \Q^\ast.
\]
Further, we assume that the operator~$\D$ is elliptic, $\A$ non-negative, $\Ce$ inf-sup stable, and that there exists a constant~$c_{\calK}>0$ such that
$c_{\calK}\|q\|_{\P} \le \|\calK q \|_{\calM^\ast} $ for all $q \in \ker \Ce =: \P_{\ker}$. The right-hand sides are of the form 
\[
  \f\colon[0,T] \to \calM^\ast, \qquad 
  \g_1\colon[0,T] \to \calH, \qquad 
  \g_2\colon[0,T] \to \P^\ast, \qquad 
  \h\colon[0,T] \to \Q^\ast.
\]
An initial value of~$p$ is then called {\em consistent} if the difference $p(0)-\Ce^- \h(0)$ in an element of the closure of $\P_{\ker}$ in $\calH$, which we denote by~$\calH_{\ker}$ in the sequel. 

The present paper considers the particular case~$\P=H^1(\bE)$, $\calH=\calM=L^2(\E)$, and $\Q=\R^\Nin$. The kernel of~$\Ce$ is given by $\P_{\ker}=H^1_0(\bE)$ and~$\calH_{\ker}$ equals~$L^2(\E)$, since $H^1_0(\bE)$ is dense in~$L^2(\E)$. For the operators we have $\A= a$ and $\D =d$, whereas for the right hand side of~\eqref{eqn_op_B_eps_a} we consider~$\g_1=\g$ and $\g_2=-\Ci^\ast \r$. We emphasize that -- apart from Lemmas~\ref{lem_wlap} and~\ref{lem_unbounded_A} -- we have never made use of the specific operators within the results of Sections~\ref{sect:formulation} - \ref{sect:timeInt}. As a consequence, all results can be generalized assuming~$p(0)$ to be consistent and~$\pstar$ to be an element of $\calH_{\ker}$ or $\P_{\ker}$, respectively.

For the generalization of Lemma~\ref{lem_wlap} one shows the ellipticity of the operator~$\wlap:= \calK^\ast \D^{-1} \calK\colon \P \to \P^\ast$ on~$\P_{\ker}$ by 
\begin{align*}
\langle \wlap \pt, \pt \rangle  = \langle  \D^{-1}\calK \pt, \D \D^{-1} \calK \pt \rangle_{\calM,\calM^\ast} \geq c_{\D} \| \D^{-1}\calK \pt \|^2_{\calM} \geq \frac{c_{\D}}{C_{\D}^2} \| \D \D^{-1} \calK \pt \|^2_{\calM^\ast} \geq  \frac{c_{\D}c_\calK}{C_{\D}^2} \| \pt \|^2_{\P}
\end{align*}
for all $\pt\in \P_{\ker}$. The included constants~$c_{\D}$ and~$C_{\D}$ are the ellipticity and continuity constants of~$\D$. In this general setting, Lemma~\ref{lem_unbounded_A} would consider the operator 
\begin{equation*}
A_\beta:=\begin{bmatrix}
- \beta \A & \calK^\ast\\
- \calK & -\D/\beta
\end{bmatrix} \colon D(A_\beta)\subset (\calH_{\ker} \times \calM) \to \calH_{\ker} \times \calM
\end{equation*}
for arbitrary positive~$\beta > 0$ and with the domain 
\[ 
  D(A_\beta) 
  = \P_{\ker} \times  \big\{ m \in \calM \,|\ \exists\, \mstar \in \calH_{\ker} \colon (\mstar, \pt) 
  = \langle \calK^* m, \pt\rangle\ \text{for all } \pt\in \P_{\ker} \big\} .
\]
Note that for every $p \in \P_{\ker}$, the term~$\A p$ is an element of~$\calH$. Since $\calH_{\ker}$ is a closed subspace of~$\calH$, there exists an element~$\hstar \in \calH_{\ker}$ with $(\hstar,\pt)= \langle \A p,\pt \rangle$ for all $\pt\in \P_{\ker}$. 
For the proof of Lemma~\ref{lem_unbounded_A}, it remains to show that~$A_\beta$ is densely defined. Since~$D(A_\beta)$ is independent of $\beta$, we may fix~$\beta =1$. Let $(h,m)\in \calH_{\ker}\times \calM$ be arbitrary. By the Gelfand triple $\P_{\ker}, \calH_{\ker}, \P_{\ker}^\ast$, there exist for every~$\eps >0$ elements $\pt_{h} \in \P_{\ker}$ and $\widetilde{\g} \in \calH_{\ker}$ with $\|\pt_{h} - h\|_{\calH} < \eps$ and $\| \widetilde{\g}  - (\A \pt_{h} - \calK^\ast m) \|_{\P_{\ker}^\ast} < \eps$. Following the steps of the proof of Lemma~\ref{lem_op_B_station}, we can show that~$-A_1$ as a mapping from $\P_{\ker} \times \calM$ to $\P_{\ker}^\ast \times \calM^\ast$ has a bounded inverse. Now, let $(\pt^{\, \prime}, m^{\prime}) \in \P_{\ker} \times \calM$ be the unique solution of 
\begin{equation*}
\begin{alignedat}{5}
	\A \pt^{\, \prime} &\ -\ &\calK^\ast m^\prime &=&&\ \widetilde{g} &&\qquad \text{in }\P^*_{\ker},\\
	\calK \pt^{\, \prime} &\ +\ &\D  m^\prime &=&&\ \calK \pt_{h} + \D  m  &&\qquad \text{in }\calM^*.
\end{alignedat}
\end{equation*}
By the construction of~$\hstar$, one shows that $(\pt^{\, \prime}, m^{\prime}) \in D(A_1)=D(A_\beta)$. We finally choose $(\pt^{\, \prime}, m^{\prime})$ as approximation of $(h,m)$ and we get with the boundedness of~$-A^{-1}_{1}$ the estimate 
\begin{align*}
  \| h-\pt^{\, \prime}\|_{\calH} + \| m -m^\prime\|_{\calM} 
  &\lesssim \| h-\pt_{h}\|_{\calH} + \|\pt_{h}-\pt^{\, \prime}\|_{\P} + \| m -m^\prime\|_{\calM} \\
  &\lesssim \| h-\pt_{h}\|_{\calH} + \| \widetilde{\g}  - (\A \pt_{h} - \calK^\ast m) \|_{\P_{\ker}^\ast} 
  < 2\, \eps.
\end{align*}

%% file: appendix_O-Bound.tex
\section{On the Sharpness of the Estimates}\label{app_sharpness}
In this section, we prove that the estimates of Theorem~\ref{thm_p0m0} and~\ref{thm_p0m0_2} are sharp in terms of powers of~$\eps$. For this, we consider problem~\eqref{eqn_op_B_eps} on a single edge of length one. The damping parameter is constant~$d=1$ and homogeneous boundary conditions are considered for the pressure~$p$. Note that under this configuration, one can rewrite the problem for $p$ as 
\begin{equation}
\label{eqn_problem_p_app_B}
 \eps \ddot p + \dot p -\partial_{xx} p =0,
\end{equation} 
with initial data $p(x,0)$ and $\dot p (x,0)=-\partial_x m(x,0)$.

First, we consider Theorem~\ref{thm_p0m0_2}, where we choose a vanishing initial condition for the pressure and 
\[
  m(x,0) = \sum_{k=1}^\infty \frac{\cos(\pi k x)}{\pi k^\alpha}, \qquad \alpha >1/2.
\]
Note that $m(x,0)$ is an~$L^2(0,1)$-function, since $\|m(x,0)\|^2_{L^2(0,1)}=\sum_{k=1}^\infty \frac{1}{2\pi^2 k^{2\alpha}}<\infty$, but that we do not have~$\partial_x p(x,0) = - m(x,0)$. Let~$\eps$ be chosen such that 
\[ 
  K(\eps) = \frac{1}{2 \pi \sqrt{\eps}} \in \N.
\]
Then the (mild) solution is given by
\begin{align*}
p(x,t;\eps)=& e^{-t/2\eps}\Bigg[\sum_{k=1}^{K(\eps)-1}  \frac{\eps\sinh\Big(\frac{t}{\eps} \sqrt{\tfrac 1 4 -\eps(\pi k)^2}\Big)}{\sqrt{\tfrac 1 4 -\eps(\pi k)^2}} k^{1-\alpha} \sin(\pi k x) + t K^{1-\alpha}(\eps) \sin(\pi K(\eps) x)\\
 &\qquad \quad + \sum_{k=K(\eps)+1}^\infty \frac{\eps\sin\Big(\frac{t}{\eps} \sqrt{\eps(\pi k)^2- \tfrac 1 4}\Big)}{\sqrt{\eps(\pi k)^2- \tfrac 1 4}} k^{1-\alpha} \sin(\pi k x)\Bigg]\\
m(x,t;\eps)=&  e^{-t/2\eps}\Bigg[\sum_{k=1}^{K(\eps)-1} \bigg\{\cosh\Big(\frac{t}{\eps} \sqrt{\tfrac 1 4 -\eps(\pi k)^2}\Big) - \frac{\sinh\Big(\frac{t}{\eps} \sqrt{\tfrac 1 4 -\eps(\pi k)^2}\Big)}{2\sqrt{\tfrac 1 4 -\eps(\pi k)^2}}\bigg\} \frac{\cos(\pi k x)}{\pi k^\alpha}\\
&\qquad \quad + \Big\{1- \frac{t}{2\eps}\Big\}\frac{\cos(\pi K(\eps) x)}{\pi K^\alpha(\eps)}\\
&\qquad \quad + \sum_{k=K(\eps)+1}^{\infty} \bigg\{\cos\Big(\frac{t}{\eps} \sqrt{\eps(\pi k)^2 - \tfrac 1 4 }\Big) - \frac{\sin\Big(\frac{t}{\eps} \sqrt{\eps(\pi k)^2 - \tfrac 1 4 }\Big)}{2\sqrt{\eps(\pi k)^2 - \tfrac 1 4 }}\bigg\} \frac{\cos(\pi k x)}{\pi k^\alpha}\Bigg]\\
\lambda(x,t;\eps)=& \begin{bmatrix}
-m(0,t;\eps) \\ m(1,t;\eps)
\end{bmatrix},
\end{align*}
which one gets by a separation of variables in~\eqref{eqn_problem_p_app_B}.
For the parabolic limit case ($\eps =0$) the solution is $p_0=m_0=\lambda_0=0$. Therefore, the norm $\|p(\cdot,\cdot\,;\eps)\|^2_{C([0,T];L^2(0,1))}$ is bounded from above by $\O(\eps)$ independently of $\alpha$, see Theorem~\ref{thm_p0m0_2}. For a lower bound, we notice that 
\begin{align*}
\|p(\cdot,\cdot;\eps)\|^2_{C([0,T];L^2(0,1))} &\geq \|p(\cdot,\eps; \eps)\|^2_{L^2(0,1)}\\
 &\geq  \frac{\eps^2}{2e} \sum_{k=K(\eps)+1}^\infty \frac{k^{2-2\alpha}}{\eps(\pi k)^2- \tfrac 1 4} \sin^2\big(\sqrt{\eps(\pi k)^2- \tfrac 1 4}\big)\\
   &= \frac{2\eps^2 K^2(\eps)}{e} \sum_{k=1}^\infty \frac{(K(\eps)+k)^{2-2\alpha}}{(K(\eps)+k)^2- K^2(\eps)} \sin^2\Big(\tfrac{1}{2}\sqrt{\tfrac{2k}{K(\eps)} + (\tfrac{k}{K(\eps)})^2}\Big)\\
   &\geq \frac{\eps}{8 e \pi^2}  \sum_{k=K(\eps)\Big\lceil\sqrt{1+(\tfrac{\pi}{3})^2}-1\Big\rceil}^{K(\eps)\Big\lfloor\sqrt{1+(\tfrac{5\pi}{3})^2}-1\Big\rfloor} \frac{1}{(K(\eps)+k)^{2\alpha}}\\
   &\geq \frac{\eps K(\eps)}{8 e \pi^2 K^{2\alpha}(\eps)}  \frac{\Big\lfloor\sqrt{1+(\tfrac{5\pi}{3})^2} -1 \Big\rfloor - \Big\lceil\sqrt{1+( \tfrac{\vphantom{5}\pi}{3})^2}  -1 \Big\rceil}{\big(\lfloor\sqrt{1+(\tfrac{5\pi}{3})^2}\rfloor\big)^{2\alpha}}\\
   &= \O(\eps^{1/2+\alpha}).
\end{align*}
The limit $\alpha \to 1/2$ then shows that the $\O(\eps)$-bound in Theorem~\ref{thm_p0m0_2} is indeed sharp. The corresponding result for Theorem~\ref{thm_p0m0} can be shown with the initial data $m(x,0)=0$ and 
\[
  p(x,0)=\sum_{k=1}^\infty \frac{\sin(\pi k x)}{k^{1+\alpha}} \in H^1(0,1).
\]
%

%% file: main_Arxiv2.bbl
\begin{thebibliography}{EKLS{\etalchar{+}}18}
	\bibitem[AH15]{AltH15}
	R.~Altmann and J.~Heiland.
	\newblock Finite element decomposition and minimal extension for flow
	equations.
	\newblock {\em ESAIM Math. Model. Numer. Anal.}, 49(5):1489--1509, 2015.
	
	\bibitem[Alt15]{Alt15}
	R.~Altmann.
	\newblock {\em {R}egularization and {S}imulation of {C}onstrained {P}artial
		{D}ifferential {E}quations}.
	\newblock Dissertation, Technische Universit{\"a}t Berlin, 2015.
	
	\bibitem[AZ18]{AltZ18}
	R.~Altmann and C.~Zimmer.
	\newblock {R}unge-{K}utta methods for linear semi-explicit operator
	differential-algebraic equations.
	\newblock {\em Math. Comp.}, 87(309):149--174, 2018.
	
	\bibitem[BCP96]{BreCP96}
	K.E. Brenan, S.L. Campbell, and L.~R. Petzold.
	\newblock {\em Numerical solution of initial-value problems in
		differential-algebraic equations}.
	\newblock Society for Industrial and Applied Mathematics (SIAM), Philadelphia,
	PA, 1996.
	
	\bibitem[BF91]{BreF91}
	F.~Brezzi and M.~Fortin.
	\newblock {\em Mixed and hybrid finite element methods}.
	\newblock Springer-Verlag, New York, 1991.
	
	\bibitem[BGH11]{BroGH11}
	J.~Brouwer, I.~Gasser, and M.~Herty.
	\newblock Gas pipeline models revisited: model hierarchies, nonisothermal
	models, and simulations of networks.
	\newblock {\em Multiscale Model. Simul.}, 9(2):601--623, 2011.
	
	\bibitem[BMXZ17]{BeaMXZ17ppt}
	C.~Beattie, V.~Mehrmann, H.~Xu, and H.~Zwart.
	\newblock Port-{H}amiltonian descriptor systems.
	\newblock Preprint 1705.09081, ArXiv e-prints, 2017.
	
	\bibitem[Bra07]{Bra07}
	D.~Braess.
	\newblock {\em Finite Elements - Theory, Fast Solvers, and Applications in
		Solid Mechanics}.
	\newblock Cambridge University Press, New York, third edition, 2007.
	
	\bibitem[EK17]{EggK17ppt}
	H.~Egger and T.~Kugler.
	\newblock An asymptotic preserving mixed finite element method for wave
	propagation in pipelines.
	\newblock Preprint 1701.04011, ArXiv e-prints, 2017.
	
	\bibitem[EK18]{EggK18}
	H.~Egger and T.~Kugler.
	\newblock Damped wave systems on networks: exponential stability and uniform
	approximations.
	\newblock {\em Numer. Math.}, 138:839--867, 2018.
	
	\bibitem[EKLS{\etalchar{+}}18]{EggKLMM18}
	H.~Egger, T.~Kugler, B.~Liljegren-Sailer, N.~Marheineke, and V.~Mehrmann.
	\newblock On structure-preserving model reduction for damped wave propagation
	in transport networks.
	\newblock {\em SIAM J. Sci. Comput.}, 40:A331--A365, 2018.
	
	\bibitem[EM13]{EmmM13}
	E.~Emmrich and V.~Mehrmann.
	\newblock Operator differential-algebraic equations arising in fluid dynamics.
	\newblock {\em Comput. Methods Appl. Math.}, 13(4):443--470, 2013.
	
	\bibitem[ET10]{EmmT10}
	E.~Emmrich and M.~Thalhammer.
	\newblock Stiffly accurate {R}unge-{K}utta methods for nonlinear evolution
	problems governed by a monotone operator.
	\newblock {\em Math. Comp.}, 79(270):785--806, 2010.
	
	\bibitem[GHS16]{GoeHS16}
	S.~G\"ottlich, M.~Herty, and P.~Schillen.
	\newblock Electric transmission lines: control and numerical discretization.
	\newblock {\em Optimal Control Appl. Methods}, 37(5):980--995, 2016.
	
	\bibitem[HLR89]{HaiLR89}
	E.~Hairer, C.~Lubich, and M.~Roche.
	\newblock {\em The Numerical Solution of Differential-Algebraic Systems by
		{R}unge-{K}utta Methods}.
	\newblock Springer-Verlag, Berlin, 1989.
	
	\bibitem[HT17]{HucT17}
	C.~Huck and C.~Tischendorf.
	\newblock Topology motivated discretization of hyperbolic {PDAE}s describing
	flow networks.
	\newblock Preprint Humboldt-Universit{\"a}t zu Berlin, 2017.
	
	\bibitem[HW96]{HaiW96}
	E.~Hairer and G.~Wanner.
	\newblock {\em Solving Ordinary Differential Equations {II}: Stiff and
		Differential-Algebraic Problems}.
	\newblock Springer, Berlin, second edition, 1996.
	
	\bibitem[JT14]{JanT14}
	L.~Jansen and C.~Tischendorf.
	\newblock A unified ({P}){DAE} modeling approach for flow networks.
	\newblock In {\em Progress in Differential-Algebraic Equations}, pages
	127--151. Springer, Berlin, Heidelberg, 2014.
	
	\bibitem[JZ12]{JacZ12}
	B.~Jacob and H.~J. Zwart.
	\newblock {\em Linear Port-{H}amiltonian Systems on Infinite-dimensional
		Spaces}.
	\newblock Springer, Basel, 2012.
	
	\bibitem[KKO99]{KokKO99}
	P.~Kokotovi{\'c}, H.~Khalil, and J.~O'Reilly.
	\newblock {\em Singular {P}erturbation {M}ethods in {C}ontrol: {A}nalysis and
		{C}ontrol}.
	\newblock Society for Industrial and Applied Mathematics (SIAM), Philadelphia,
	PA, corrected reprint edition, 1999.
	
	\bibitem[KM06]{KunM06}
	P.~Kunkel and V.~Mehrmann.
	\newblock {\em Differential-Algebraic Equations: Analysis and Numerical
		Solution}.
	\newblock European Mathematical Society (EMS), Z{\"u}rich, 2006.
	
	\bibitem[LMT13]{LamMT13}
	R.~Lamour, R.~M{\"a}rz, and C.~Tischendorf.
	\newblock {\em Differential-Algebraic Equations: A Projector Based Analysis}.
	\newblock Springer-Verlag, Berlin, Heidelberg, 2013.
	
	\bibitem[LO95]{LubO95}
	C.~Lubich and A.~Ostermann.
	\newblock Runge-{K}utta approximation of quasi-linear parabolic equations.
	\newblock {\em Math. Comp.}, 64(210):601--627, 1995.
	
	\bibitem[Meh13]{Meh13}
	V.~Mehrmann.
	\newblock Index concepts for differential-algebraic equations.
	\newblock In T.~Chan, W.J. Cook, E.~Hairer, J.~Hastad, A.~Iserles, H.P.
	Langtangen, C.~{Le Bris}, P.L. Lions, C.~Lubich, A.J. Majda, J.~McLaughlin,
	R.M. Nieminen, J.~Oden, P.~Souganidis, and A.~Tveito, editors, {\em
		Encyclopedia of Applied and Computational Mathematics}. Springer-Verlag,
	Berlin, 2013.
	
	\bibitem[MWTA00]{MagWTA00}
	P.~C. Magnusson, A.~Weisshaar, V.~K. Tripathi, and G.~C. Alexander.
	\newblock {\em Transmission Lines and Wave Propagation}.
	\newblock Taylor \& Francis, fourth edition, 2000.
	
	\bibitem[Osi87]{Osi87}
	A.~Osiadacz.
	\newblock {\em Simulation and analysis of gas networks}.
	\newblock Gulf Pub. Co., London, 1987.
	
	\bibitem[Paz83]{Paz83}
	A.~Pazy.
	\newblock {\em Semigroups of Linear Operators and Applications to Partial
		Differential Equations}.
	\newblock Springer-Verlag, New York, 1983.
	
	\bibitem[Rou05]{Rou05}
	T.~Roub\'{\i}{\v c}ek.
	\newblock {\em Nonlinear Partial Differential Equations with Applications}.
	\newblock Birkh{\"a}user Verlag, Basel, 2005.
	
	\bibitem[{van}13]{Van13}
	A.~J. {van der Schaft}.
	\newblock {\em Port-{H}amiltonian Differential-Algebraic Systems}, pages
	173--226.
	\newblock Springer, Berlin, Heidelberg, 2013.
	
	\bibitem[VR18]{VouR18}
	I.~Voulis and A.~Reusken.
	\newblock Discontinuous {G}alerkin time discretization methods for parabolic
	problems with linear constraints.
	\newblock Preprint 1801.06361, ArXiv e-prints, 2018.
	
	\bibitem[Wlo92]{Wlo92}
	J.~Wloka.
	\newblock {\em Partial Differential Equations}.
	\newblock Cambridge University Press, Cambridge, 1992.
	
	\bibitem[Zei90]{Zei90a}
	E.~Zeidler.
	\newblock {\em Nonlinear Functional Analysis and its Applications {IIa}: Linear
		Monotone Operators}.
	\newblock Springer-Verlag, New York, 1990.
\end{thebibliography}
